\documentclass[notitlepage,11pt,reqno]{amsart}

\usepackage{amsopn,esint,nicefrac}
\usepackage[final]{hyperref}
\usepackage[T1]{fontenc}

\usepackage[utf8]{inputenc}
\usepackage{apptools}
\AtAppendix{\counterwithin{lemma}{section}} 
\usepackage[margin=1in]{geometry}
\allowdisplaybreaks
\newcommand{\stkout}[1]{\ifmmode\text{\sout{\ensuremath{#1}}}\else\sout{#1}\fi}

 \newcommand{\grad}{\triangledown}

\newcommand{\I}{\mathcal I}

\newcommand{\xbar}{\bar{x}}
\newcommand{\ybar}{\bar{y}}

\newcommand{\Rn}{\mathbb{R}^n}

\usepackage{graphicx,enumitem,dsfont,upgreek}
 \usepackage[dvips]{epsfig}
 \usepackage[mathscr]{eucal}
\usepackage{amscd}
\usepackage{amssymb,nicefrac}
\usepackage{amsthm}
\usepackage{amsmath}
\usepackage{latexsym}
\usepackage{dsfont}
\usepackage{upref}
\usepackage{hyperref}

\usepackage{color}
\theoremstyle{plain}

\newtheorem{thm}{Theorem}[section]
\theoremstyle{plain}
\newtheorem{lem}[thm]{Lemma}

\newtheorem{cor}[thm]{Corollary}

\newtheorem{defi}[thm]{Definition}
\newtheorem{rem}{Remark}[section]
\theoremstyle{definition}
\newtheorem*{maintheorem*}{Main Theorem}
\newtheorem*{maincorollary*}{Main Corollary}
{%
\setcounter{enumi}{0}

\begin{enumerate}}%
{\end{enumerate} }

{%
\setcounter{enumi}{0}

\begin{enumerate}}%
{\end{enumerate} }

\newcommand{\norm}[1]{\ensuremath{\left\|#1\right\|}}
\newcommand{\abs}[1]{\ensuremath{\left|#1\right|}}

\newcommand{\cC}{\ensuremath{\mathcal{C}}}

\newcommand{\cI}{\ensuremath{\mathcal{I}}}

\newcommand{\dist}{{\rm dist}}

\newcommand{\R}{\ensuremath{\mathbb{R}}}

\newcommand{\Rd}{\ensuremath{\R^d}}

\newcommand{\dz}{\ensuremath{\, dz}}

\numberwithin{equation}{section} \allowdisplaybreaks

\usepackage{cancel,pdfsync}

\title[Nonlocal Liouville theorems]{Nonlocal Liouville theorems with gradient nonlinearity}

\begin{document}

\author{Anup Biswas}

\address{Indian Institute of Science Education and Research-Pune, Dr.\ Homi Bhabha Road, Pashan, Pune 411008. INDIA Email:
{\tt anup@iiserpune.ac.in}}

\author{Alexander Quaas}

\address{
Departamento de Matem\'atica,
Universidad T\'ecnica Federico Santa Mar\'{i}a
Casilla V-110, Avda.
Espa\~{n}a, 1680 – Valpara\'{i}so, CHILE.
Email: {\tt alexander.quaas@usm.cl}}

\author{Erwin Topp}

\address{
Instituto de Matem\'aticas, Universidade Federal do Rio de Janeiro, Rio de Janeiro - RJ, 21941-909, BRAZIL; and Departamento de Matem\'atica y C.C., Universidad de Santiago de Chile, Casilla 307, Santiago, CHILE.
Email: {\tt etopp@im.ufrj.br; erwin.topp@usach.cl.}
}

\begin{abstract}
In this article we consider a large family of nonlinear nonlocal equations involving gradient nonlinearity and provide a
unified approach, based on the Ishii-Lions type technique, to establish Liouville properties of the solutions. We also answer
an open problem raised by  \cite{CG21}. Some applications to regularity issues are also studied.
\end{abstract}


\maketitle

\section{Introduction}
The main goal of this article is to study Liouville properties of viscosity solutions to nonlinear nonlocal equations of the form
$$-\cI u + H(u, \grad u)=0\quad \text{in}\; \Rn,$$
where the Hamiltonian $H(r, p)$ are models of the form 
\begin{align*}
& \mbox{(I)}  \ |p|^m \quad \mbox{for} \ m>1, \\
& \mbox{(II)}  \ -r^q|p|^m, \\
& \mbox{(III)}  \ r^q|p|^m,
\end{align*} 
and $\cI$ is a fractional Pucci type nonlinear
operator of order $2s, s\in (0, 1)$. 

The Liouville property of the solutions to various nonlinear partial differential equations has been a central topic of research not only because of its own importance, but also 
due to its applications in the regularity theory (see \cite{BK23,RS-Duke,RS16,Ser15}), a priori bounds \cite{BDGQ} etc. This particular article is inspired from
the recent survey of Cirant and Goffi \cite{CG21} where the authors provide a comprehensive account on  the development of various Liouville theorems for classical (or local) elliptic equations. The nonlocal analogue of most of these Liouville theorems remained open because of several difficulties stemming from the nonlocal nature of the equations and unavailability of appropriate gradient bounds, also known as Bernstein estimate.
In this article we propose a unified approach to prove  Liouville properties for a large class of nonlinear nonlocal operators including gradient
nonlinearity of type I, II and III above.

For $H=0$, the above Liouville property is quite well studied and have also been used to establish interior regularity of solutions \cite{Fall16,RS-Duke,RS16,Ser15}. The case $H(u, \grad u)=-u^q$ is known as the Lane-Emden equation and has been extensively studied 
in literature, see \cite{Biswas,CL09,QX16,XY13} and references therein. Recently, Barrios and Del Pezzo \cite{BD20} study Liouville property 
of positive solutions to the equations $(-\Delta)^s u + |\grad u|^m=\lambda u^q$ in exterior domains. The authors rely on the maximum principles 
to obtain an appropriate decay estimate of the solutions at infinity which helps them to establish the required results. This particular idea in the local case goes back to the work of Alarc\'{o}n et.\ al. \cite{AGQ}, see also the earlier results by Cutr\`i and Leoni \cite{CL00} and Armstrong and Sirakov in \cite{AS}; and in the context of nonlocal equation see Felmer and the second author \cite{FQ} and the references therein. We should also mention the works of Alibaud et.\ al.\cite{ATEJ}, Berger and Schilling \cite{BS22}, Fall and Weth \cite{FW16} which study Liouville properties for a general linear integro-differential operator with constant coefficients. In fact, \cite{ATEJ, BS22} provides a necessary and sufficient condition for the Liouville property to hold in the class of bounded solutions.  We also remark on higher order nonlocal Liouville theorems that can be found in \cite{DSV19}, where the authors introduce a notion of fractional Laplacian evaluated at functions with arbitrary polynomial growth at infinity.

To understand the difficulty in treating problems considered in this article, let us focus on the model I. That is, the Hamiltonian is of the form
$H(u, \grad u)=|\grad u|^m$. For the classical second-order Laplacian, that is $\cI = \Delta$, the Liouville property is obtained from the gradient estimate.
More precisely, since $v(x):=R^{-\frac{m-2}{m-1}}u(Rx), R>0,$ satisfies the same equation, from the gradient estimate it follows that
$$|\grad u(x)|\leq C R^{-\frac{1}{m-1}}\quad x\in B_{R/2}.$$
Letting $R\to\infty$, we see that $u={\rm const}$. This idea dates back to the work of Peletier and Serrin \cite{PS78}, Lions \cite{Lions}.
Later, a similar gradient estimate is used to obtain the Liouville property for model II, see Bidaut-V\'eron, Garc\'ia-Huidobro and V\'eron \cite{BGV}, 
Filippucci, Pucci and Souplet \cite{FPS20}. Therefore, the gradient estimate plays a key role in establishing the Liouville property. 
Since a solution to the stationary equation also solves the associated parabolic equation, the Liouville property can be established by leveraging the regularizing effect of the parabolic model. In this context, we refer to the work of Porretta and Priola  \cite{PP13}, 
who used an Ishii-Lions type argument to establish global Lipschitz regularity for solutions to the parabolic equations, and subsequently applied this to derive a Liouville result for the stationary equations. See also \cite{CGM24} where the authors employed semiconcavity estimates in a similar manner to establish a Liouville theorem for first-order parabolic Hamilton-Jacobi equations.
Surprisingly,
this Bernstein type gradient estimate remained a challenging problem for the  nonlocal operators (see \cite{CDV,Goffi}). 
Using the Bernstein approach, Cabr\'e, Dipierro and Valdinoci \cite{CDV} derived a gradient estimate for solutions to a class of Pucci-type operators, which does not involve gradient nonlinearity. However, the kernels in \cite{CDV} have to be regular enough for their Bernstein method to be applicable.
Very recently, the first and third author in~\cite{BT23}
obtained a gradient estimate by employing an Ishii-Lions type arguemnt that, although useful, appears to be far from optimal and is only applicable to model I.  When combined with an appropriate scaling, the gradient bound in~\cite{BT23} yields a Liouville property for the solution to model I (see Remark~\ref{R-GL} below). 
As observed, the presence of a nonlocal term in the gradient estimate presents a significant challenge when trying to adapt the scaling arguments from local settings to nonlocal operators. At this point, let us also mention the work of 
Constantin and Vicol \cite{CV12}, who used a Bernstein-type argument to establish well-posedness for SQG equations. Their approach relies on a nonlinear maximum principle, which is applicable to bounded functions with vanishing gradients at infinity.

 Our method in this article is based on Ishii-Lions type argument which was introduced by
Ishii and Lions in \cite{IL90} to obtain H\"{o}lder regularity of viscosity solutions for degenerate elliptic second-order equations.
In the same second-order setting, Capuzzo-Dolcetta, Leoni, and Porretta [19] obtained interior Lipschitz estimates by revisiting Bernstein's technique, incorporating elements of the Ishii-Lions method (see also [5]). In the nonlocal setting, the Ishii-Lions argument was implemented by Barles et al. [6] to establish Hölder and Lipschitz estimates for a broad class of nonlinear nonlocal equations. Typically, this method involves a doubling variables procedure and the use of a penalization function, which serves as a test function for the solution. An appropriate selection of a large Lipschitz constant for this test function, along with the ellipticity of the equation, leads to the desired result.

We apply a similar Ishii-Lions method but with a penalization function with a small H\"{o}lder seminorm (among other technicalities involved by the nonlocal nature of the equation) and then, making use of the ellipticity, we establish the results. Interestingly, this technique works for a large class of nonlinear equations and, therefore, produces 
a unified approach in obtaining the Liouville property. As it can be seen from the proofs below, a similar technique can also be applied to nonlocal operators with kernels of variable order and more general integro-differential
operators with non-degenerate second order elliptic term. In the later part of this article, we also apply Liouville theorems to establish certain regularity results. For instance, in Theorem~\ref{T3.1} we treat a nonlinear
equation with critical diffusion with respect to the gradient (namely, fractional order $s=\frac{1}{2}$) and show that the viscosity solutions are $C^\gamma$ for any $\gamma\in (0,1)$. This extends the result of Schwab and Silvestre \cite{SS16},
who obtained $C^\alpha$ regularity for some $\alpha \in (0, 1)$.

We finish this introduction by mentioning that the Ishii-Lions method we employ here does not seem to be effective when applied to inequalities.
See for instance, the proof of Lemma~2.4. In contrast, such Liouville properties are well-established for both the classical Laplacian and the
$p$-Laplacian in the case of  inequalities. In this context, we can refer to the integral method, developed by several researchers  (cf.\ \cite{CDM08,CM97,MP04,RS01}) 
which has been successfully used to establish the Liouville property for inequalities. However, similar techniques for nonlocal operators, particularly those with gradient nonlinearity, remain largely unexplored.

The rest of the article is organized as follows: in the next Section we introduce our model and state all our main results on Liouville property. Section~\ref{S-proof} is devoted to the proof of these
results. Finally, in Section~\ref{S-appl} we provide two applications of our Liouville theorems to obtain some regularity results.

\subsection{Setting and main results}
Given ellipticity constants $0 < \lambda \leq \Lambda$, we denote by $\mathcal{K}_{s} = \mathcal K_s(\lambda, \Lambda)$ the family of nonnegative, measurable kernels $K : \R^d \setminus \{ 0 \} \to \R$ for which the following condition holds
\begin{equation}\label{elliptic}
\frac{\lambda}{|z|^{n + 2s}} \leq K(z) \leq \frac{\Lambda}{|z|^{n + 2s}}, \quad z \neq 0.
\end{equation}
Note that we do not assume $K$ to be symmetric with respect to $0$.

For each $K \in \mathcal K_s$, we denote
\begin{equation*}
L_K u(x) = \int_{\R^d} [u(x + z) - u(x) - 1_{\{2s\in [1, 2)\}}1_{B}(z) \nabla u(x) \cdot z] K(z)dz,
\end{equation*}
which is well-defined for adequately regular functions $u: \R^n \to \R$.  Here $B$ denotes the unit ball in $\Rn$ around $0$. Let $\cI$
be a translation invariant, positively $1$-homogeneous, nonlocal operator which is elliptic with respect to the class $\mathcal{K}_s$, that is, 
\begin{align}\label{ellipticity}
\cI u_h(x)&=\cI u (x+h)\quad \text{for all}\; x, h\in\Rn,\quad \text{where}\; u_h:= u(\cdot+h),\nonumber
\\
\inf_{K\in \mathcal{K}_s} (L_K u(x)-L_Kv(x))&\leq \cI u(x)-\cI v(x) \leq \sup_{K\in\mathcal{K}_s} (L_K u(x)- L_K v(x))
\end{align}
whenever $u, v\in C^{1,1}(x)\cap L^1(\omega_s)$ where $\omega_s(z)=(1+|z|)^{-n-2s}$. Note that for $u\in C^{1,1}(x)\cap L^1(\omega_s)$ and
$v\in C^{1,1}(y)\cap L^1(\omega_s)$, letting $\tilde{v}=v(\cdot + y-x)$, we also have from \eqref{ellipticity} that
\begin{equation}\label{ellip1}
\inf_{K\in \mathcal{K}_s} (L_K u(x)-L_Kv(y))\leq \cI u(x)-\cI v(y) \leq \sup_{K\in\mathcal{K}_s} (L_K u(x)- L_K v(y)).
\end{equation}

We also denote by $\mathcal{K}_{\rm sym} = \mathcal K_{\rm sym}(\lambda, \Lambda)$ the subclass of kernels in $\mathcal K_s$ which are symmetric, that is, $K(z) = K(-z)$ for all $z \neq 0$. Some of our results would use 
a translation invariant, positively $1$-homogeneous, nonlocal operator 
$\cI_{\rm sym}$ which is elliptic with respect to the class $\mathcal{K}_{\rm sym}$. Our results below are developed for viscosity solutions and are therefore also applicable to classical solutions.

\smallskip

Now we are ready to state our first main result.
\begin{thm}\label{T1.1}
Let $s\in (\frac{1}{2}, 1)$ and $H:\Rn\to\R$ be a continuous function satisfying the following:
for every $\varepsilon, L>0$ there exists $C=C(\varepsilon, L)$
satisfying
\begin{equation}\label{ET1.1A}
|H(p)-H(q)|\leq C\, |p-q|
\end{equation}
for all $\varepsilon \leq |p|, |q|\leq L$. Consider a viscosity solution $u$ to
\begin{equation}\label{ET1.1B}
-\cI u + H(\grad u)=0\quad \text{in}\; \Rn.
\end{equation}
If for some $\gamma\in [0, \frac{1}{2s})$ and $M>0$ we have 
\begin{equation}\label{growthu}
|u(x)|\leq M(1+|x|)^\gamma \quad x \in \Rn,
\end{equation}
then $u$ is necessarily a constant. 

\smallskip

Furthermore, for $s\in (0, \frac{1}{2}]$ the same conclusion holds, provided $\gamma\in [0, 2s)$ and one of the following holds:
\begin{itemize}
    \item[(i)] $u$ is  uniformly continuous in $\Rn$ and $H$ is locally Lipschitz in the sense that 
    for every $L>0$ there exists $C=C( L)$
satisfying
\begin{equation}\label{ET1.1A1}
|H(p)-H(q)|\leq C\, |p-q|\quad \text{for all}\; |p|, |q|\leq L.
\end{equation}
    \item[(ii)] $H$ is globally Lipschitz.
\end{itemize}
\end{thm}

\begin{rem}
It is interesting to note that we do not require $H$ to be locally Lipschitz for \eqref{ET1.1A} to hold. Also, under the stated conditions
of Theorem~\ref{T1.1},  \eqref{ET1.1B} does not have any solution unless $H(0)=0$. Again, for $s>1/2$, we can not take the growth exponent 
$\gamma$ to be $1$. For instance, if we consider the equation
$$(-\Delta)^s u + |e\cdot \grad u|^m=0 \quad \text{in}\; \Rn, \; m>0$$
for some unit vector $e$, then $u(x)=e^\perp\cdot x$ is a non-zero solution for any unit vector $e^\perp$ satisfying $e^\perp\cdot e=0$.
\end{rem}

Theorem~\ref{T1.1} deals with a quite general Hamiltonian in the sense that we do not impose any sign or growth conditions on $H$.
If we impose a slightly restrictive condition on $H$ in Theorem~\ref{T1.1}, the Liouville property holds for a larger
class of functions, as stated in the following theorem
\begin{thm}\label{T1.2}
Let $s>1/2$ and $H:\Rn \to \R$ be a locally Lipschitz function
satisfying 
\begin{equation}\label{HT1.2}
|H(p)-H(q)|\leq C_H (|p|^{m-1} + |q|^{m-1})|p-q|\quad \text{for all}\; p, q\in\Rn,
\end{equation}
where $m>1$ and $C_H>0$. Let $\cI_{\rm sym}$ be a translation invariant, positively $1$-homogeneous operator which is 
elliptic with respect to the set of all symmetric kernels $\mathcal{K}_{\rm sym}$.
Consider the equation
$$
-\cI_{\rm sym} u + H(\grad u)=0 \quad \text{in}\; \Rn.
$$
Then there is no non-constant solution $u$
satisfying 
$$|u(x)|\leq M(1+|x|)^\gamma, \quad x\in\Rn$$
for some $\gamma< \max\{\frac{m-2s}{m-1}, \frac{1}{2s}\}$.
\end{thm}

\begin{rem}\label{R-GL}
The bound of $\gamma$ in Theorem~\ref{T1.2} can be slightly improved for a particular class of equations.
Let $u$ be a viscosity solution to $(-\Delta)^s u + |\grad u|^m=0$ in $\Rn$ for $m\geq 2s>1$. Suppose that $|u|\lesssim (1+|x|^{\theta})$ for
$\theta=\frac{m-2s}{m-1}$. Letting $v_R(x)=R^{-\theta} v(Rx)$ for $R>1$, we see that $|v_R(x)|\leq \kappa (1+|x|^\theta)$, 
where $\kappa$ does not depend on $R$,
and $v_R$ also solves the same equation. Applying the gradient 
estimate of \cite[Theorem~2.1]{BT23} we get that
$$|v_R(x)-v_R(0)|\leq \kappa_1 \quad \text{for all}\; |x|\leq 1,$$
where $\kappa_1$ does not depend on $R$. Scaling back to $u$ this gives
$$ |u(x)-u(0)|\leq \kappa_1 R^{\theta-1} |x|\quad \text{for}\; |x|\leq R.$$
Letting $R\to\infty$, we find that $u$ is a constant.
\end{rem}

It is also interesting to ask whether there exists a solution that is comparable to $|x|^\gamma$ near infinity. Our next result gives an answer to this question.
\begin{thm}\label{T1.5}
Suppose that $u$ is a non-trivial solution to 
\begin{equation}\label{ET1.5A}
-\cI_{\rm sym} u + |\grad u|^m=0\quad \text{in}\; \Rn
\end{equation}
for $m>1$ and $s>\frac{1}{2}$, that satisfies
\begin{align}
|u(x)|&\leq M(1+|x|)^\gamma, \quad \text{for some $\gamma<2s$ and all $x \in \R^n$},\label{ET1.5B}
\\
\limsup_{|x|\to\infty} \frac{|u(x)|}{|x|^\gamma}&>0.\label{ET1.5C}
\end{align}
Then $\gamma\leq \max\{0, \frac{m-2s}{m-1}\}$. 
\end{thm}

Now we start with our second model. In \cite{FPS20}, Filippucci, Pucci and Souplet consider the equation $\Delta u + u^q|\grad u|^m=0$ in $\Rn$ for $m>2, q>0$, and
it is shown that  any positive bounded classical solution to this equation is a  constant. One of the main ingredients of their proof is a
Bernstein type estimate on $u^{\frac{q+m-1}{m-2}}$ which is used to establish that $u^{q+1}|\grad u|^{m-2}$ is bounded in $\Rn$. 
Interestingly, our proof of Theorem~\ref{T1.1} can be modified suitably to get our next result which extends the Liouville property 
\cite[Theorem~1.1]{FPS20} to the nonlocal setting.

\begin{thm}\label{T1.3}
Suppose that $s>\frac{1}{2}$. Then there exists no bounded positive viscosity solution to
$$-\cI u = u^q|\grad u|^m\quad \text{in}\; \Rn,$$
for $m>2s, q>0$, unless $u$ is a constant.
\end{thm}

Our next result considers equations with coercive Hamiltonian, that is, model (III). In case of local operators, such problems are studied in \cite{FS11,FPR,F09}. More precisely, these works deal with a general family of nonlinear variational operators
of $p$-Laplacian type.  In particular, Filippucci, Pucci and Rigoli \cite{FPR}
show that the equation
$$\Delta_p u \geq u^q|\grad u|^m\quad \text{in}\; \Rn, \quad \text{with}\quad 0\leq m <p-1,$$
has only trivial $C^1$ solutions if and only if $q+1>p-m$. Farina and Serrin \cite{FS11} study Liouville properties for similar equations without imposing any sign condition on the $C^1$ solutions.
In this article we prove the following Liouville property
\begin{thm}\label{T1.4}
Let $s\geq 1/2$ and $u$ be a non-negative viscosity solution to
\begin{equation}\label{ET1.4A}
-\cI u  + u^q|\grad u|^m=0\quad \text{in}\; \Rn,
\end{equation}
for some $q\geq 0, m \geq 0, q+m>1,$ and
$$|u(x)|\leq M (1+|x|)^\gamma $$
for some $\gamma< 1$. Then $u$ is necessarily a constant. In particular, in view of Theorem~\ref{T1.5}, there exists no non-negative solution to \eqref{ET1.5A} satisfying \eqref{ET1.5B}-\eqref{ET1.5C}, other than constants.
\end{thm}
Theorems~\ref{T1.2}, ~\ref{T1.5} and ~\ref{T1.4}  partly answers the open problem mentioned in \cite[Open problem 6.6]{CG21}.

The proofs of Filippucci, Pucci and Rigoli \cite{FPR} rely on the construction of appropriate radial solution by the method of ode and then application of maximum principles whereas the arguments
of Farina and Serrin \cite{FS11} use variational techniques. None of these methods seems to be useful in proving Theorem~\ref{T1.4}.

\begin{rem}
The conclusion of Theorem~\ref{T1.4} also holds for a more general equation of the form
$$-\cI u  + u^q|\grad u|^m= u^{-\delta}\quad \text{in}\; \Rn,$$
for some $\delta\geq -1$ and with the same conditions on $q$ and $m$. Note that for $\delta\in [-1, 0)$ and $q=0$, this produces a Liouville property not covered by \cite{BD20}.
\end{rem}

The technique developed in this article can also deal with equations with singular nonlinearity. In fact, we could prove the following Liouville property

\begin{thm}\label{T1.6}
Let $f:(0, \infty)\to \R$ be a continuous non-increasing function. Then there exists no non-constant positive viscosity solution to the equation
\begin{equation}\label{ET1.6A}
-\cI_{\rm sym} u = f(u)\quad \text{in}\; \Rn,
\end{equation}
provided
\begin{equation}\label{ET1.6A1}
|u(x)|\leq C(1+|x|)^\gamma
\end{equation}
and one of the following holds
\begin{itemize}
\item[(i)] $\gamma<1\wedge 2s$;
\item[(ii)] $s>1/2$ and
$f$ is a  $C^1$ convex function satisfying
\begin{equation}\label{Cnd-1}
\limsup_{R\to\infty} R^{2s-(2-\beta)\gamma}|f'(R^\gamma)|^{\frac{2s-1}{2s}}=\infty,
\end{equation}
for some $\beta \in (0, 1)$.
\end{itemize}
Here $\cI_{\rm sym}$ is a translation invariant, positively $1$-homogeneous operator which is 
elliptic with respect to the set of all symmetric kernels $\mathcal{K}_{\rm sym}$.
\end{thm} 

The following corollary is immediate from the above theorem.
\begin{cor}
The equation $(-\Delta)^s u = u^{-p}, p>0,$ does not have a solution satisfying \eqref{ET1.6A1} provided
one of the following holds.
\begin{itemize}
\item[(i)] $2s\leq 1$ and $\gamma<2s$;
\item[(ii)] $2s>1$ and either $\gamma<1$ or 
$$2s>\gamma \left[1+(p+1)\frac{2s}{2s-1}\right].$$
\end{itemize}
\end{cor}
When $\cI_{\rm sym}$ is replaced by the Laplacian operator, the proof follows by taking a logarithmic transformation on $u$ and then applying
the Bernstein estimate (see \cite{Yang}). This method does not apply to the nonlocal setting as explained before in the introduction.
 It should also be noted that for local operators, the Liouville property typically holds without requiring any growth assumptions on the solution.  However, for the $s$-fractional Laplacian, and aside the definition of fractional Laplacian introduced in~\cite{DSV19}, growth assumptions become necessary due to the integrability constraints imposed by the nonlocal kernel (see \cite[Section~2.4.5]{FRRO}).
Additionally, as shown in our results, we impose further growth condition on the solutions, which depend on the specific structure of the operators. In fact, in the notation of \cite{FRRO}, $|u|\lesssim (1+\abs{x})^\gamma$ is equivalent to $u\in L^\infty_\gamma(\Rn)$.

\subsection*{Notations:} $\kappa, \kappa_1, \ldots, C, C_1, \ldots$ are arbitrary constants whose value may change from line to line.
$B_r(x)$ will denote the ball of radius $r$ around $x$. For $x=0$, we simply denote this ball by $B_r$ whereas the unit ball around $0$ is denoted by $B$.

\section{General Strategy and Technical Lemmas.}

In this section, we sketch the central idea for proving our results. We start with the notion of solution, for which we introduce a general equation, 
 allowing the discussion in a general context. 


Let $\Omega \subseteq \R^n$, $U \subseteq \R$ be open sets, $H \in C(\R^n \times U \times \R^n)$ and  $\I$ be a nonlocal operator as 
mentioned in the introduction. We consider Hamilton-Jacobi equation of the form
\begin{align}\label{eqgral}
-\I u + H(x, u, \nabla u) = 0 \quad \mbox{in} \ \Omega.
\end{align}
By a solution to this equation we would mean viscosity solution which we recall for convenience. 
\begin{defi}\label{def-sol}
We say $u: \R^n \to U$ measurable, $u \in L^1(\omega_s)$, upper semicontinuous (lower semicontinuous) in $\bar\Omega$, is a viscosity subsolution (supersolution) to~\eqref{eqgral} in $\Omega$ if, for each $x_0 \in \Omega$ and  $\varphi \in C^2(B_r(x_0))$ for some $r > 0$ satisfying $u(x_0) = \varphi(x_0)$ and $u \leq \varphi$ (resp. $u \geq \varphi$) in $B_r(x_0)$, we have
$$
- \cI v(x_0) + H(x_0, u(x_0), \grad \varphi(x_0)) \leq \ (resp. \ \geq) \ f(x_0), 
$$
where
\begin{equation}\label{eval-u}
v(y)=\left\{\begin{array}{ll}
\varphi(y) & \mbox{for}\; y\in B_r(x_0),
\\[2mm]
u(y) & \text{elsewhere}.
\end{array}
\right.
\end{equation}
We say that $u$ is a viscosity solution to problem~\eqref{eqgral} in $\Omega$ if it is both sub and supersolution in the above sense.

In the remaining part of the paper, we refer $v$ in~\eqref{eval-u} as a test function for $u$ with smooth representative $\varphi$ at $x_0$.  
\end{defi}

Related to viscosity evaluations, we notice that in the case of subsolutions $u$, it is enough if  $x_0$ is a point of maximum of $u - \varphi$ in $B_r(x_0)$. 
This follows from the observation that if we have two functions $v_1, v_2$ such that $v_2 = v_1 + c$ in $B_r(x_0)$ for some $c \in \R$, then from \eqref{ellipticity}, we have
\begin{equation*}
\I v_1(x_0) = \I v_2(x_0).
\end{equation*}
This will allow us to use a relaxed version of viscosity subsolution. The same observation can be made for supersolutions, and hence for solutions.

\subsection{General Strategy.}\label{Sec-GS}
 Since the proofs involve technical calculations, for the ease of reading
we first explain the common strategy  to prove Theorems~\ref{T1.1}-\ref{T1.4}, and Theorem~\ref{T1.6}. These arguments are inspired from the celebrated Ishii-Lions method~\cite[Section VII]{IL90}
and have been adapted to the nonlocal setting following the ideas of Barles et.\ al.\ \cite{Barles-12}.

Let us begin by providing a rough sketch on our line of argument for the model problem
$$(-\Delta)^s u + |\grad u|^m=0.$$
Set $\delta\in (0, 1)$. Suppose that $(x, y)\mapsto u(x)-u(y)-\delta|x-y|$ attains a positive maximum at the point $(\bar{x},\bar{y})$. Note that this is may not
be possible without additional penalization. Letting $\phi_{\bar x}(z)=\delta|\bar{x}-z|$ and $\phi_{\bar y}(z)=\delta|z-\bar{y}|$, it can be 
easily seen that
\begin{align*}
u(\bar{x}+z)-u(\bar{x}) & \leq \phi_{\bar y}(\bar{x}+z)-\phi_{\bar y}(\bar{x}), \\
u(\bar{y}+z)-u(\bar{y}) & \geq -\phi_{\bar x}(\bar{y}+z)+\phi_{\bar x}(\bar{y}) \\
u(\bar{x}+z)-u(\bar{x})-(u(\bar{y}+z)-u(\bar{y})) & \leq 0.
\end{align*}

Since, $(-\Delta)^s u(\bar{x})-(-\Delta)^s u(\bar{y})=0$, the above inequalities gives 
\begin{equation}\label{try}
(-\Delta)^s_{\mathcal C}\phi_{\bar y}(\bar{x}) + (-\Delta)^s_{\mathcal C}\phi_{\bar x}(\bar{y})\leq 0
\end{equation}
where
$$
(-\Delta)^s_{\mathcal C} f(x) :=\int_{\mathcal C} (f(x+z)+f(x-z)-2f(x))\frac{dz}{|z|^{n+2s}},
$$
and $\mathcal C$ is any measurable subset of $\mathbb R^n$.
Note that the gradient nonlinearity cancels out since there is no localization, and we assumed the existence of a maximizer
$(\bar{x},\bar{y})$.
Now, we choose ${\mathcal C}$ in such a way that the left-hand side of \eqref{try} becomes positive, leading to a contradiction. Hence we must
have $u(x)-u(y)\leq \delta |x-y|$ for all $x, y$ and $\delta\in (0, 1)$, which can only be true if $u$ is constant.

Now we make above argument mathematically viable. We consider a function $\varphi \in C([0,+\infty))$, $\varphi(0) = 0$, so that $\varphi$ is smooth, increasing and concave in some interval $(0, a_0)$ for some $a_0 > 0$. This function $\varphi$ could be understood as a modulus of continuity. For $R > 0$, we also require a function $\tilde \chi_R: \R^n \to \R$, smooth and nonnegative with $\tilde \chi_R = 0$ in $B_R$,  such that 
\begin{equation}\label{limsup0}
\limsup_{|x| \to +\infty} \frac{|u(x)|}{\tilde \chi_R(x)} < 1.
\end{equation}
This function plays the role of the localization. With these ingredients, take $\delta \in (0,1)$ and define the function $\Phi: \R^n \times \R^n \to \R$ as
$$
\Phi(x,y) = u(x) - u(y) - \delta \varphi(|x - y|) - \tilde \chi_R(x) - \tilde \chi_R(y), \quad x, y \in \R^n.
$$

Interestingly, the above doubling argument shares a deep connection with a certain probabilistic coupling 
\cite[Appendix]{PP13}.
Our main goal is to prove that for all $\delta \in (0,1)$ small enough, there exists $R_0 = R_0(\delta) > 1$ large enough, such that for all $R \geq R_0$ we have
\begin{equation}\label{E1.6}
\sup_{\R^n \times \R^n} \Phi(x,y) \leq 0.
\end{equation}
Once this inequality is established, it can be easily seen that $u$ is locally constant. In fact, for each compact set $K \subset \R^n$ and each $\delta$, we can find $R \geq R_0$ large enough so that $K \subset B_R$,
implying that, for each $x, y \in K$ we have
$$
u(x) - u(y) - \delta \varphi(|x - y|) \leq \sup_{\R^n \times \R^n} \Phi(x,y) \leq 0.
$$
In particular,
$$
|u(x) - u(y)| \leq \delta \varphi(|x - y|) \quad \mbox{for all} \ x,y \in K,
$$
and since $\delta$ is arbitrary, we conclude that $u$ is locally constant, and hence constant in $\R^n$.

The proof of~\eqref{E1.6} is carried out by contradiction, that is, assuming that:
\begin{equation}\label{supPhigral}
\mbox{for each $\delta > 0$ small and $R$ large enough, we have} \ \sup_{\R^n \times \R^n} \Phi(x,y) > 0.
\end{equation}

By the choice of $\chi_R$, the supremum is attained at some point $(\bar x, \bar y) \in \R^n \times \R^n$, and since $\varphi, \chi_R$ are nonnegative, we also have $\bar x \neq \bar y$. We denote 
\begin{equation*}
\phi(x, y) = \delta \varphi(|x - y|) + \tilde \chi_R(x) + \tilde \chi_R(y), \quad x, y \in \R^n.
\end{equation*}
Therefore,
\begin{equation}\label{testing}
\left \{ \begin{array}{l}
\bar x \  \mbox{is a point of global maximum  of} \ x \mapsto u(x) - (u(\bar y) + \phi(x, \bar y)), 
\\[2mm]
\bar y \ \mbox{is a point of global minimum of} \ y \mapsto u(y) - (u(\bar x) - \phi(\bar x, y)),
\end{array} \right .
\end{equation}
and since $\bar x \neq \bar y$, we can test the equation at $\bar x$ treating $u$ as a subsolution, and at $\bar y$ treating $u$ as a supersolution.
This leads us to the set of inequalities
\begin{align*}
-\I u(\bar x) + H(\bar x, u(\bar x), \nabla u(\bar x)) & \leq 0, \\
-\I u(\bar y) + H(\bar y, u(\bar y), \nabla u(\bar y)) & \geq 0,
\end{align*}
understood in the viscosity sense. Subtracting these inequalities, we get that
$$
I \leq I_H,
$$
where
\begin{align*}
I = & -\I u(\bar x) + \I u(\bar y), \\
I_H = & -H(\bar x, u(\bar x), \nabla u(\bar x)) + H(\bar y, u(\bar y), \nabla u(\bar y)).
\end{align*}
We show that for suitable choice of $\delta$ and $R$ we must have $I_H << I$ which contradicts the above inequality and therefore, \eqref{supPhigral} does not hold.

\smallskip

From this point, we need to consider the particularities of each problem to estimate both the terms $I$ and $I_H$. Though already treated in the literature, the estimates for the nonlocal terms are less standard and because of this we provide next some technical lemmas that are going to be useful for our proofs.

\subsection{Technical Lemmas.} This section consists of three lemmas which are used repeatedly in our proofs.  
We need a few notations to state these lemmas. For $K \in \mathcal K_s$, $A \subset \R^n$ measurable, $f: \R^n \to \R$ measurable, and $x, \xi \in \R^n$, we denote
$$
L_K[A] (f, x, \xi) = \int_A (f(x+z)-f(x)-{\bf 1}_{B} \xi \cdot z) K(z)dz,
$$
whenever the integral makes sense. This is the case, for instance, when  $f \in C^2$ in a neighborhood of $x$, $\xi = \nabla f(x)$, and $f\in L^1(\omega_s)$. In that case, we simply write 
$$
L_K[A]f(x) = L_K[A](f, x, \nabla f(x)).
$$
For technical reasons, we also consider the notation
\begin{equation}\label{tildeLK}
\tilde L_K[A] (f, x, \xi) = \int_A (f(x+z)-f(x)- \xi \cdot z) K(z)dz,
\end{equation}
that is, when the compensator term acts in the whole domain of integration. As before, when $f$ is smooth we write
$$
\tilde L_K[A] f(x) = \tilde L_K[A](f, x, \nabla f(x)).
$$
The first one, Lemma~\ref{est-cone}, encodes the contribution coming from the ellipticity of the operator $\I$.
\begin{lem}\label{est-cone}
Let $\bar a \in \R^n, \bar a \neq 0$, and for $\eta_0, \delta_0 \in (0, 1)$, consider the cone
\begin{equation}\label{cone}
\cC = \cC_{\delta_0, \eta_0}:=\{z\in B_{\delta_0|\bar{a}|}\; :\; | \bar{a}\cdot z|\geq (1-\eta_0)| \bar{a}||z|\}. 
\end{equation}

Assume $\varphi \in C([0, +\infty))$  with $\varphi \in C^2(0, |\bar a|(1 + \delta_0))$ increasing and concave, and denote $\tilde \varphi(x) = \varphi(|x|)$ for $x \in \R^n$. Then, for each $K$ we have
\begin{equation*}
\tilde L_K[\cC] (\tilde \varphi, \bar a) \leq  \frac{1}{2} \int_\cC \sup_{0\leq t\leq 1}
\left\{(1-\tilde\eta^2)\frac{\varphi'(|\bar{a}+tz|)}{(1-\delta_0)|\bar{a}|} +\tilde\eta^2\,\varphi^{\prime\prime}(|\bar{a}+tz|)\right\}|z|^2 K(z) \dz,
\end{equation*}
where
\begin{equation*}
\tilde\eta=\frac{1-\eta_0-\delta_0}{1+\delta_0}.
\end{equation*}
\end{lem}

\begin{proof}
Since $z\mapsto\varphi(|\bar a +z |)$ is $C^2$ in the cone $\cC$, by Taylor's  expansion we have
\begin{align*}
&\varphi(|\bar a+z|)-\varphi(|\bar{a}|)-z\cdot \varphi'(|\bar{a}|)\frac{\bar{a}}{|\bar{a}|}
 \\
\leq & \ \frac{1}{2}\sup_{0\leq t\leq 1}
\left\{\frac{\varphi'(|a+tz|)}{|\bar{a}+tz|}(|z|^2-\langle \widehat{\bar{a}+tz}, z\rangle^2) +\varphi^{\prime\prime}(\bar{a}+tz)\langle \widehat{\bar{a}+tz}, z\rangle^2\right\},
\end{align*}
for all $z\in B_{\delta_0|\bar{a}|}$, where $\hat{\zeta}$ denotes the unit vector along $\zeta\neq 0$. Thus, for each $K\in\mathcal{K}_s$, we obtain
\begin{align}\label{AB003}
&\tilde{L}_K[\cC] \tilde \varphi (\bar{a}) \leq \frac{1}{2} \int_\cC \sup_{0\leq t\leq 1}
\left\{\frac{\varphi'(|\bar a+tz|)}{|\bar a+tz|}(|z|^2-\langle \widehat{\bar a+tz}, z\rangle^2) +\varphi^{\prime\prime}(\bar a+tz)\langle \widehat{\bar a+tz}, z\rangle^2\right\} K(z) \dz.
\end{align}
For $z\in\cC_{\delta_0,\eta_0}$ and $|t|\leq 1$ we have
\begin{equation*}\label{AB001}
 |\langle \bar{a} + t z, z\rangle|\geq (1-\eta_0-\delta_0) |\bar{a}||z|,
\end{equation*}
and
\begin{equation}\label{AB002}
(1-\delta_0)|\bar{a}|\leq |\bar{a}+tz|\leq (1+\delta_0)|\bar{a}|.
\end{equation}
Using these inequalities in \eqref{AB003} we get the result.
\end{proof}

We require the following estimate for the nonlocal operator evaluated on the localization function.
\begin{lem}\label{est-localization}
Let $\chi_0: [0, +\infty) \to [0, 1]$ be a smooth, nondecreasing function satisfying $\chi_0(t) = 0$ if $t < 1/8$, $\chi_0(t) = 1$ if $t \geq 1/4$. For  $\alpha \in [0, 2s)$ and $R > 1$, define
\begin{equation}\label{defchi}
\chi(x) = \chi_R(x) = \chi_0(|x|/R) (1 + |x|^2)^{\alpha/2}, \quad x \in \Rn.
\end{equation}
Then, there exists a constant $C$ such that
for each $A > 1$ and $K \in \mathcal K_s$, we have
\begin{equation*}
|L_K \chi(x)| \leq 
\left\{\begin{array}{ll}
C A^\alpha R^{\alpha - 1} & \text{for}\; s\in (1/2, 1),
\\
C A^\alpha R^{\alpha - 1}\log R & \text{for}\; s=1/2,
\\
C A^\alpha R^{\alpha - 2s} & \text{for}\; s\in (0, 1/2),
\end{array}
\right.
\qquad \mbox{for all} \ |x| \leq AR,
\end{equation*}
and if $K \in \mathcal{K}_{\rm sym}$, then we have
\begin{equation*}
|L_K \chi(x)| \leq C A^\alpha R^{\alpha - 2s}, \quad \mbox{for all} \ |x| \leq AR.
\end{equation*}

\end{lem}

\begin{proof}
Denoting $\psi(x)=(1+|x|^2)^{\frac{\alpha}{2}}$, we see that
\begin{align*}
\partial_{x_i} \chi(x) = & R^{-1}\chi_0'(|x|/R) \psi(x)\frac{x_i}{|x|} + \chi_0(|x|/R) \partial_{x_i} \psi(x), \\
\partial_{x_ix_j}^2\chi(x)
= & (R^{-2}\chi^{\prime\prime}_0(|x|/R)\frac{x_ix_j}{|x|^2}+R^{-1}\chi^\prime_0(|x|/R)(\delta_{ij}{|x|}^{-1}-x_ix_j|x|^{-3}) \psi(x) \\
& +  R^{-1}\chi^\prime_0(|x|/R)\left(\frac{x_i}{|x|} \partial_{x_j}\psi(x)+\frac{x_j}{|x|} \partial_{x_i}\psi(x)\right)
+ \chi_0(|x|/R)\partial_{x_ix_j}^2\psi(x).
\end{align*}
Now we estimate $L_K \chi(x)$ by dividing the domain in parts. Using that $\chi_0(|x|/R)=0$ for $|x|\leq R/8$, $\chi_0', \chi_0''$ are uniformly bounded, and $\alpha < 2s < 2$, there exists $\kappa > 0$, not depending on $R$ and $A$, such that
\begin{equation*}
\| D^2 \chi \|_{L^\infty(\R^n)} \leq \kappa R^{\alpha - 2},
\end{equation*}
giving us, for all $K\in\mathcal{K}_s$, that
\begin{equation*}
|L_K[B] \chi(x)| \leq \kappa_1 \Lambda R^{\alpha - 2},
\end{equation*}
for some $\kappa_1 > 0$, not dependent on $R$ and $A$. Also, this estimate is uniform in $x$.

Now consider $|x| \leq AR$. There exists $C > 0$ (not dependent on $R$ and $A$) such that 
\begin{equation*}
\| \nabla \chi \|_{L^\infty(B_R(x))} \leq C A^\alpha R^{\alpha-1}, 
\end{equation*}
which gives us
\begin{align*}
|L[B_R\setminus B]\chi(x)|
 \leq C A^\alpha R^{\alpha-1} \int_{B_R \setminus B} |z| K(z)dz   \leq 
 \left\{
\begin{array}{ll}
C \Lambda A^\alpha R^{\alpha-1} & \text{for}\; s>1/2,
\\[2mm]
C \Lambda A^\alpha R^{\alpha-1} \log R & \text{for}\; s=1/2,
\\[2mm]
C \Lambda A^\alpha R^{\alpha-2s} & \text{for}\; s<1/2.
\end{array}
\right.
\end{align*}
for all $K\in\mathcal{K}_s$. 
Finally, there exists $C > 0$, not depending on $A, R$, such that
\begin{align*}
|L[B_R^c] \chi(x)| \leq C\int_{B_R^c} ((AR)^\alpha + |z|^\alpha)K(z)dz \leq C \Lambda A^{\alpha} R^{\alpha - 2s}
\end{align*}

Thus, gathering the above estimates, we conclude the result for $K \in \mathcal K_s$.

\medskip

When $K\in\mathcal{K}_{\rm sym}$, we can write
\begin{align*}
L_K \chi(x) = \tilde L_K [B_R] \chi(x) + L_K[B_R^c] \chi(x),
\end{align*}
and  using the estimate for $D^2 \chi$, we get
$$
|\tilde L_K [B_R] \chi(x)| \leq C R^{\alpha - 2} \int_{B_R}|z|^2 K(z)dz \leq C R^{\alpha - 2s},
$$
for some $C > 0$, not depending on $R$ and $A$. Again, for the integral on $B_R^c$, we use the same estimate as in the previous case. This concludes the proof.
\end{proof}

\begin{lem}\label{nonlocal-IJ}
Let $u_1, u_2 \in C(\R^n) \cap L^1(\omega_s)$, $\varphi$ as in Lemma~\ref{est-cone}, and $\chi_1, \chi_2 \in C^2(\R^n) \cap L^1(\omega_s)$. Assume there exists $(\bar x, \bar y) \in \R^n$ with $\bar x \neq \bar y$, global maximum point of the function
$$
(x,y) \mapsto u_1(x) - u_2(y) - \phi(x,y), \quad x,y \in \R^n,
$$ 
where $\phi(x,y) = \varphi(|x -y|) - \chi_1(x) - \chi_2(y)$. 

Denote $v_1$ the test function for $u_1$ with smooth representative $\phi(\cdot, \bar y)$ at $\bar x$, and $v_2$ the test function of $u_2$  with smooth representative $-\phi(\bar x, \cdot)$ at $\bar y$. Then, setting $\bar a = \bar x - \bar y$ and considering $\cC$ as in~\eqref{cone} for some $\delta_0, \eta_0 \in (0,1)$, the following inequality holds
\begin{align*}
-\I v_1(\bar x) + \I v_2(\bar y) &\geq -2\sup_{K \in \mathcal K_s} \{ \tilde L_K[\cC] \tilde \varphi(\bar a)\} - \sup_{K \in \mathcal K_s} \{ L_K[\cC\cup B^c_r] \chi_1(\bar x)\} - 
\sup_{K \in \mathcal K_s} \{ L_K[\cC\cup B^c_r] \chi_2(\bar y)\}
\\
&\quad -\sup_{K\in\mathcal{K}_s} |L_K[\cC^c \cap B_r] v_1(\bar x)| -\sup_{K\in\mathcal{K}_s} |L_K[\cC^c \cap B_r] v_2(\bar y)|,
\end{align*}
where,  $\tilde \varphi(x) = \varphi(|x|)$ as in Lemma~\ref{est-cone},  $\tilde L_K$ is given by \eqref{tildeLK}, and $B_r$ is the ball given by \eqref{eval-u}. In particular,
\begin{align*}
-\I v_1(\bar x) + \I v_2(\bar y) &\geq -2\sup_{K \in \mathcal K_s} \{ \tilde L_K[\cC] \tilde \varphi(\bar a)\} - \sup_{K \in \mathcal K_s} \{ L_K \chi_1(\bar x)\} - 
\sup_{K \in \mathcal K_s} \{ L_K \chi_2(\bar y)\} +  \Theta(r),
\end{align*}
for some function $\Theta(r)$, dependent on $\phi, \chi_1, \chi_2$, satisfying $\lim_{r\to 0}\Theta(r)=0$.

The same result holds for $\I_{\rm sym}$, elliptic with respect to the family $\mathcal K_{\rm sym}$, replacing $\mathcal K_s$ by $\mathcal K_{\rm sym}$ in the above inequalities.
\end{lem}

\begin{proof}
Denote $\hat{a}=\frac{1}{|\bar{a}|}\bar{a}$ and $\Phi(x,y) = u_1(x) - u_2(y) - \phi(x,y)$ for $x, y \in \R^n$.

Fix $K\in\mathcal{K}_s$. Using the fact  that
$$
\Phi(\xbar+z, \ybar)\leq\Phi(\xbar,\ybar)\Rightarrow
u_1(\xbar+z)-u_2(\xbar)\leq \phi(|\xbar+z-\ybar|)-\phi(|\xbar-\ybar|),
$$
we have
\begin{align}\label{ET1.1C0}
L_K[\cC] v_1 (\xbar) \leq L_K[\cC] \phi(\cdot, \ybar)(\xbar).
\end{align}

Similarly, the relation $\Phi(\xbar,\ybar +z)\leq\Phi(\xbar,\ybar)$ would lead to
\begin{align}\label{ET1.1D0}
L_K[\cC] v_2 (\ybar) \geq 
L_K[\cC] (-\phi(\xbar,\cdot))(\ybar).
\end{align}

Since
$$
\nabla_x\phi(\xbar,\ybar)=\varphi'(|\bar{a}|)\hat{a} + \nabla \chi_1(\xbar)
\quad \text{and}\quad 
-\nabla_y\phi(\xbar,\ybar)=\varphi'(|\bar{a}|)\hat{a} - \nabla \chi_2(\ybar),
$$
we have that
\begin{align*}
L_K[\cC]\phi(\cdot, \bar y)(\bar x) = & \int_{\cC} [\varphi(|\bar a + z|) - \varphi(|\bar a|) - 1_{B}(z) \varphi'(|\bar a|) \hat a \cdot z] K(z)dz + L_K[\cC] \chi_1 (\bar x), \\
L_K[\cC](-\phi(\bar x, \cdot))(\bar y) = & - \int_{\cC} [\varphi(|\bar a - z|)  - \varphi(|\bar a|) + 1_{B}(z) \varphi'(|\bar a|) \hat a \cdot z] K(z)dz - L_K[\cC] \chi_2(\bar y).
\end{align*}

Using \eqref{ET1.1C0} and \eqref{ET1.1D0}, we conclude that
\begin{align*}
L_K[\cC] v_1(\bar x) - L_K[\cC] v_2(\bar y) \leq & \int_{\cC} [\varphi(|\bar a + z|) - \varphi(|\bar a|) - 1_B(z) \varphi'(|\bar a|) \hat a \cdot z] K(z)dz \\
& +  \int_{\cC} [\varphi(|\bar a - z|) - \varphi(|\bar a|) + 1_B(z) \varphi'(|\bar a|) \hat a \cdot z] K(z)dz \\
& + L_K[\cC]  \chi_1 (\bar x) + L_K[\cC]  \chi_2(\bar y).
\end{align*}
Writing $K^*(z) = K(-z)$ for $z \neq 0$, we see that
\begin{equation*}
L_K[\cC] v_1(\bar x) - L_K[\cC] v_2(\bar y) \leq \tilde L_K[\cC] \tilde \varphi(\bar a) + \tilde L_{K^*}[\cC] \tilde \varphi(\bar a) + L_K[\cC] \chi_1 (\bar x) + L_K[\cC] \chi_2(\bar y).
\end{equation*}
Notice that if $K \in \mathcal K_s$, then $K^* \in \mathcal K_s$ too, and if $K \in \mathcal K_{\rm sym}$, then $K^* = K$.
Also, note that $\Phi(\xbar+z,\ybar+z)\leq \Phi(\xbar,\ybar)$ for all $z$ implies
$$
(u_1(\xbar+z)-u_1(\xbar)) - (u_2(\ybar+z)-u_2(\ybar))
\leq (\chi_1(\xbar + z)- \chi_1(\xbar)) + (\chi_2(\ybar + z)-\chi_2(\ybar)),
$$
from which we readily have 
\begin{equation*}
L_K[\cC^c] v_1(\bar x) - L_K[\cC^c] v_2(\bar y) \leq L_K[\cC^c\cap B^c_r] \chi_1(\bar x) + L_K[\cC^c\cap B^c_r] \chi_2(\bar y)
+ |L_K[\cC^c\cap B_r] v_1(\bar x)| + |L_K[\cC^c\cap B_r] v_2(\bar y)|.
\end{equation*}

Gathering the above inequalities, and using the ellipticity of $\I$ (see \eqref{ellip1}), we see that
\begin{equation*}\label{I1I2I3}
\begin{split}
-\I v_1(\bar x) + \I v_2(\bar y) \geq &  - \sup_{K \in \mathcal K_s} \{ \tilde L_K v_1(\xbar) - \tilde L_{K} v_2(\bar y) \} 
\\
\geq  & -2 \sup_{K \in \mathcal K_s} \{ \tilde L_K[\cC] \tilde \varphi(\bar a) \} -  \sup_{K \in \mathcal K_s} \{ L_K[\cC\cup B_r^c] \chi_1(\bar x) \} - \sup_{K \in \mathcal K_s} \{ L_K[\cC\cup B_r^c] \chi_2(\bar y) \}
\\
&\quad - \sup_{K\in\mathcal{K}_s} |L_K[\cC^c \cap B_r] v_1(\bar x)| -\sup_{K\in\mathcal{K}_s} |L_K[\cC^c \cap B_r] v_2(\bar y)|,
\end{split}
\end{equation*}
which concludes the proof for $\mathcal K_s$. The proof for $\mathcal K_{\rm sym}$ is a straightforward adaptation.
\end{proof}

\section{Proofs of main results}\label{S-proof}
In this section, we prove our main results. As mentioned before,
the central idea of our proofs relies on the Ishii-Lions method \cite[Section~VII]{IL90} which was introduced to prove 
H\"{o}lder regularity of classical viscosity solutions. We begin with the proof of Theorem~\ref{T1.1} and most of the other proofs (other than
Theorems~\ref{T1.4} and~\ref{T1.5}) will be variation of
this particular one.
\begin{proof}[Proof of Theorem~\ref{T1.1}]
To simplify some computations, up to modifying $M$, we will use~\eqref{growthu} as
$$
|u(x)| \leq M |x|^\gamma \quad \mbox{for} \ |x| > 1\quad \text{and}\quad \max_{|x|\leq 1}|u(x)|\leq M.
$$

We consider the different cases separately.

\smallskip

\noindent
{\bf 1. Case $s \in (1/2, 1)$:} We divide the proof in four steps.

\medskip
\noindent
\textsl{Step 1: Setting the General Strategy.} For $\theta\in (0, 1)$ to be fixed later and  $R > 4$, 
we define $\varphi = \varphi_R : [0,+\infty) \to [0, +\infty)$ as
\begin{equation}\label{defvarphi}
\varphi(t)=t-\frac{1}{4R^\theta} t^{1+\theta}, \quad t \in [0, R],
\end{equation}
and extended as a constant function equal to $\varphi(R)$ on $[R, +\infty)$.
Take $\alpha \in (\gamma, 2s)$ and consider the function $\chi = \chi_R$ as in \eqref{defchi} above. Let
$$
\tilde \chi_R(x) = M2^{1 + 2\alpha} R^{\gamma - \alpha}\chi(x), \quad x \in \R^n.
$$ 
Since $\tilde \chi_R(x)\geq M2^{1 + 2\alpha} R^{\gamma - \alpha} |x|^\alpha$ for $|x|>R/4$ and $\alpha>\gamma$,
\eqref{limsup0} holds.

Now consider the function
$$
\Phi(x,y) := u(x)-u(y)-\delta\varphi(|x-y|)- (\tilde \chi_R(x) + \tilde \chi_R(y) ), \quad x, y \in \R^n,
$$
and as mention in the Section~\ref{Sec-GS}, our aim is to prove that~\eqref{E1.6} holds.

By contradiction, we assume~\eqref{supPhigral} holds, namely, there exists $\delta > 0$ and a sequence $R \to +\infty$ such that 
\begin{equation*}\label{ineqcontra}
\sup_{x, y \in \Rn} \Phi(x, y) > 0.
\end{equation*}
By our choice of localization function $\tilde\chi$, the above supremum is attained at some point $(\bar x, \bar y) \in \R^n \times \R^n$, 
and also $\bar x \neq \bar y$. From now on, we denote 
$$
\bar a = \bar x - \bar y.
$$
Notice that for  $|x|>\frac{R}{4}$ we can write
\begin{align*}
|u(x)|- M 2^{1+2\alpha} R^{\gamma - \alpha}\chi(x) &\leq M|x|^\gamma- M 2^{1+2\alpha} R^{\gamma - \alpha} |x|^\alpha
\\
&\leq |x|^\gamma M (1-2^{1+2\alpha} R^{\gamma - \alpha} |x|^{\alpha-\gamma})
\\
&\leq |x|^\gamma M (1-2^{1+2\gamma})\leq -(R/4)^\gamma M,
\end{align*}
Thus,
for $|x|>R/4$ and $y \in \R^n$, we have
\begin{align*}
 \Phi(x, y) &\leq -(R/4)^\gamma M -u(y) - \tilde \chi_R(y)\leq 0.
\end{align*}
A similar analysis can be performed for $|y| > R/4$, giving us $\Phi(x, y)\leq 0$ for $|y|>R/4$, $x \in \R^n$. 
Therefore, we must have
$$
|\bar x|, |\bar y| \leq R/4.
$$

%

Denoting 
$$
\phi(x,y) = \delta\varphi(|x-y|) + (\tilde \chi_R(x) + \tilde \chi_R(y)), \quad (x,y) \in \R^n \times \R^n,
$$
we have $x \mapsto  \phi(x, \bar y)$ is a test function for $u$ at $\bar x$, and $y \mapsto  - \phi(\bar x, y)$ is a test function for $u$ at $\bar x$ (c.f.~\eqref{testing}). Then, denoting $v_1$ the test function of $u$ with smooth representative $\phi(\cdot, \bar y)$ at $\bar x$, and $v_2$ the test function of $u$ with smooth representative $-\phi(\bar x, \cdot)$ at $\bar y$ , (c.f. ~\eqref{eval-u}), we have
\begin{align*}
-\cI v_1(\xbar) + H(\nabla_x \phi(\bar x, \bar y)) &\leq 0,
\\
-\cI v_2 (\ybar)+ H(-\nabla_y \phi(\bar x, \bar y)) &\geq 0.
\end{align*}
Then, subtracting both the inequalities, we arrive at 
\begin{equation}\label{testing0}
I \leq I_H,
\end{equation}
where
\begin{align*}
I := & -\cI v_1(\xbar) + \cI v_2(\ybar), \\
I_H := & -H(\nabla_x \phi(\bar x, \bar y)) + H(-\nabla_y \phi(\bar x, \bar y)).
\end{align*}
Now we estimate $I$ and $I_H$ separately.

\medskip
\noindent
\textsl{Step 2: Estimate of $I_H$.} A simple computation shows that for $t\leq R$ we have
\begin{equation}\label{estvarphi}
\frac{1}{2}\leq (1-\frac{1+\theta}{4R^\theta}t^\theta)=\varphi'(t) \leq 1,\quad \varphi^{\prime\prime}(t)=-\frac{(1+\theta)\theta}{4R^\theta}t^{\theta-1}.
\end{equation}
Furthermore,  there exists a constant $\kappa > 0$, independent of $R$, such that
\begin{equation*}
|\nabla \chi(\bar y)|, |\nabla \chi(\bar x)| \leq \kappa (1+R^{\alpha - 1}),
\end{equation*}
giving us 
$$
|\grad \tilde \chi_R(\xbar)|, |\grad \tilde \chi_R(\ybar)|
\leq \kappa M (R^{\gamma-1} + R^{\gamma-\alpha}),
$$
for some constant $\kappa > 0$.

Since $\gamma<1\wedge\alpha$, the l.h.s. of the above display tends to $0$ as $R\to\infty$, and therefore, combining with \eqref{estvarphi}
we see that 
$$
\frac{\delta}{8}\leq |\grad_x\phi(\xbar,\ybar)|, |\grad_y\phi (\xbar,\ybar)|\leq 1 + 2\kappa M,
$$
for all large $R$. Thus, we can apply \eqref{ET1.1A} with the local Lipschitz constant $C_\delta$, depending on $\delta$, from which we arrive at
\begin{equation}\label{estH}
I_H \leq C_\delta (R^{\gamma - 1}+ R^{\gamma - \alpha}),
\end{equation}
for all large $R$.

\medskip
\noindent
\textsl{Step 3: Estimate of $I$.} For convenience,  we recall the cone
$$
\cC = \cC_{\delta_0, \eta_0}:=\{z\in B_{\delta_0|\bar{a}|}\; :\; |\bar a\cdot z|\geq (1-\eta_0)|\bar a||z|\}\subset B_{\frac{R}{2}}
$$
for $\delta_0, \eta_0 \in (0,1)$ to be fixed,
Using Lemma~\ref{nonlocal-IJ}, we have 
\begin{align}\label{I1I2I3}
I \geq &  
 -2\delta \sup_{K \in \mathcal K_s} \{ \tilde L_K[\cC] \tilde \varphi(\bar a) \} -  \sup_{K \in \mathcal K_s} \{ L_K \tilde \chi_R(\bar x) \} - \sup_{K \in \mathcal K_s} \{ L_K \tilde \chi_R(\bar y) \}
+\Theta(r) \nonumber
 \\
=: & -2\delta I_1 - I_2 - I_3+\Theta(r),
\end{align}
where $\Theta(r)\to 0$ as $r\to 0$, for each fixed $\delta$ and $R$.
Now it remains to estimate $I_1, I_2$ and $I_3$ above.  Using Lemma~\ref{est-localization}, obtain at once that
\begin{equation}\label{I2I3}
I_2, I_3 \leq C M R^{\gamma-1}.
\end{equation}
Thus we deal with $I_1$. For this, we start by noticing that

$$
\frac{\delta}{2}|\bar a|\leq \delta \varphi(|\bar a|)\leq u(\xbar)-u(\ybar)\leq 2M(R^\gamma+1), 
$$
from which, we also get that
\begin{equation}\label{ET1.1G}
|\bar a| \leq \frac{CM}{\delta} R^\gamma,
\end{equation}
for some constant $C$. 

Now, we choose 
\begin{equation}\label{choiceetadelta1}
\eta_0 = \tau \frac{|\bar a|^\theta}{R^\theta}, \quad \delta_0 = \tau \frac{|\bar a|^\theta}{R^\theta}  < 1, 
\end{equation}
for some $\tau \in (0,1)$ small enough so that 
$$
\frac{1-\tilde\eta^2}{1-\delta_0}< 4(\eta_0 + 2\delta_0)\quad \text{and} \quad \tilde\eta^2\geq \frac{1}{2}\; .
$$ 
Then, using \eqref{AB002}, \eqref{estvarphi} together with Lemma~\ref{est-cone}, we have
\begin{equation*}
\tilde L_K[\cC] \tilde \varphi(\bar a) \leq \frac{1}{2} \frac{|\bar a|^{\theta-1}}{R^\theta} \sup_{K\in\mathcal{K}} \int_\cC
[ 12 \tau  - \frac{\theta (\theta + 1)}{8}(1+\delta_0)^{\theta-1} ]|z|^2 K(z)\dz.
\end{equation*}
We can choose $\tau$ depending on $\theta$ to conclude the existence of $c_\theta \in (0,1)$ such that
\begin{equation*}
\tilde L_K[\cC] \tilde \varphi(\bar a) \leq -c_\theta \lambda \frac{|\bar a|^{\theta-1}}{R^\theta} \int_{\cC}
|z|^2 \frac{\dz}{|z|^{n+2s}}.
\end{equation*}
To estimate the integral on the cone we use \cite[Example 1]{Barles-12} to obtain
\begin{equation*}
\tilde L_K[\cC] \tilde \varphi(\bar a) \leq -\tilde{c}_\theta \lambda \frac{|\bar{a}|^{\theta-1}}{R^\theta}\eta_0^{\frac{n-1}{2}}
(|a|\delta_0)^{2-2s},
\end{equation*}
and from the definition of $\delta_0, \eta_0$ this leads to
\begin{align*}
\tilde L_K[\cC] \tilde \varphi(\bar a) & \leq - \tilde{c}_\theta \tau^{\frac{n-1}{2}+2-2s} \lambda \frac{|a|^{\theta-1}}{R^\theta}(|\bar a|/R)^{\theta \frac{d-1}{2}}(|\bar{a}|^{1+\theta} R^{-\theta})^{2-2s}\nonumber
 \\
&=  - \tilde{c}_\theta \tau^{\frac{n-1}{2}+2-2s} \lambda R^{-\theta((n + 1)/2 + 2 - 2s)} |\bar a|^{\theta((n + 1)/2 + 2 - 2s) + 1 - 2s}\nonumber
\\
&= - \tilde{c}_\theta \tau^{\frac{n-1}{2}+2-2s} \lambda R^{-\theta\upbeta} |\bar a|^{\theta\upbeta + 1 - 2s},\label{AB005}
\end{align*}
where $\upbeta=(n + 1)/2 + 2 - 2s)$. Since $2s>1$, we now set $\theta\in (0, \frac{2s-1}{\upbeta})$ so that the exponent of $|\bar{a}|$ becomes
negative. By definition of $I_1$ and \eqref{ET1.1G} we then obtain
\begin{equation}\label{est-LK-cone}
2\delta I_1 = 2\delta \sup_K \{ \tilde L_K[\cC] \tilde \varphi(\bar a) \} \leq - \kappa_3\lambda  \delta^{2s-\theta\upbeta} R^{\theta\upbeta(\gamma-1) + \gamma(1-2s) } ,
\end{equation}
for some constant $\kappa_3$ not dependent on $\delta$ and $R$. Combining this estimate together with~\eqref{I2I3} and replacing them into~\eqref{I1I2I3}, we conclude that
\begin{equation}\label{estI1}
I \geq \kappa_4 \delta^{2s-\theta\upbeta} R^{\theta\upbeta(\gamma-1) + \gamma(1-2s) } - CMR^{\gamma-1}+ \Theta(r).
\end{equation}
for some constant $\kappa_4$, not dependent on $R$ and $\delta$.

\medskip
\noindent
\textsl{Step 4: Conclusion.} Placing this last estimate for $I$ and \eqref{estH} into~\eqref{testing0} and letting $r\to 0$, we obtain
\begin{equation}\label{ET1.1I}
\kappa_4 \delta^{2s-\theta\upbeta} R^{\theta\upbeta(\gamma-1) + \gamma(1-2s) }\leq C(R^{\gamma - \alpha}+R^{\gamma-1}) 
\end{equation}
for some constant $\kappa_4 , C$, not dependent on $R$, and for all large $R$. 

Now we can set $\alpha > 1$ so that $\alpha-\gamma > 1- \gamma $, and $\theta\in (0, \frac{2s-1}{\upbeta})$ small enough to satisfy
$$ 
\theta\upbeta(\gamma-1) + \gamma(1-2s) > \gamma-1.
$$
This can be done because $2s\gamma<1$. Then, taking $R \to +\infty$ we arrive at a contradiction to~\eqref{ET1.1I}, confirming \eqref{E1.6} holds. As explained in our General Strategy (see Section~\ref{Sec-GS}), this implies $u$ is locally constant, and therefore the result.

\smallskip

\noindent
{\bf 2. Case $s \in (0, 1/2]$:} We start with  (i). Suppose that $u$ is not constant and for some $\eta>0$ there
exist points $\tilde{x}, \tilde{y}$ such that
$$
u(\tilde{x})-u(\tilde{y})>\eta.
$$
As before, for $\delta\in (0, 1)$ and $R>4$ we let $\tilde \chi_R(x) = M 2^{1+2\alpha} R^{\gamma - \alpha} \chi(x)$ with $\chi$ as in~\eqref{defchi} with $\alpha\in (\gamma, 2s)$. We also denote
$
\varphi(t)= t^\beta 
$
for $\beta \in (\gamma, 2s)$ fixed.
We consider the function
\begin{equation}\label{AB0004}
\Phi(x,y) = u(x)-u(y)-\delta\varphi(|x-y|)- (\tilde \chi_R(x) + \tilde \chi_R(y) ), \quad x, y \in \R^n.
\end{equation}
Let us fix $R$ large enough so that
$|\tilde{x}|+|\tilde{y}|<R/8$. Set $\hat\delta=\hat\delta(\eta, \beta)$ small enough so that
$$u(\tilde x)-u(\tilde y)-\delta\varphi(|\tilde x-\tilde y|)>\eta/2\quad \text{for all}\; \delta\in (0, \hat\delta).$$
Using the definition of $\chi_0$ we note that
$$
\sup_{\Rn\times\Rn}\Phi\geq \sup_{B_R\times B_R}\Phi\geq u(\tilde{x})-u(\tilde{y}) - \delta |\tilde x - \tilde y|^\beta >\eta/2,
$$
and since~\eqref{limsup0} holds, the supremum is attained at some $(\bar x, \bar y) \in \R^n \times \R^n$. Denote $\bar a = \bar x - \bar y$.
Let $\omega$ be the modulus of continuity
of $u$ and since $u$ is not a constant, we necessarily have $\omega(t)>0$ for $t>0$. By maximality, we see that
$$
\eta/2\leq \Phi(\xbar,\ybar)\leq u(\xbar)-u(\ybar)\leq \omega(|\bar a|)
\Rightarrow |\bar a|\geq \omega^{-1}(\eta/2).
$$
In particular, $|\bar a|$ is bounded from below, independent of $R$ and $\delta$. Hence $\varphi'(|\bar{a}|)\leq \beta |\omega^{-1}(\eta)|^{\beta-1}$.

Now, letting $\phi(x,y) = \delta \varphi(|x - y|) + \tilde \chi_R(x) + \tilde \chi_R(y)$, and denoting $v_1$ the test function for $u$ with smooth representative $\phi(\cdot, \bar y)$ at $\bar x$, and $v_2$ the test function for $u$ with smooth representative $-\phi(\bar x, \cdot)$ at $\bar y$, we use the viscosity inequalities and arrive at~\eqref{testing0}, the notation being the same as before, just the assumptions need to be specified. 

First, by~\eqref{ET1.1A1}, similarly as in~\eqref{estH}, we get
\begin{equation*}
I_H \leq C_\eta R^{\gamma - \alpha},
\end{equation*}
for some constant $C_\eta > 0$ that depends on $\eta$ but not on $\delta$ nor $R$.

Using Lemmas~\ref{est-localization} and \ref{nonlocal-IJ} , we obtain
\begin{equation*}
I \geq
\left\{\begin{array}{ll}
 -2\delta \sup_{K \in \mathcal K_s} \{ \tilde L_K \tilde \varphi(\bar a) \} + \Theta(r) - C R^{\gamma - 2s}  & \text{for}\; s\in (0, 1/2),
 \\[2mm]
 -2\delta \sup_{K \in \mathcal K_s} \{ \tilde L_K \tilde \varphi(\bar a) \} + \Theta(r) - C R^{\gamma - 1}\log R & \text{for}\; s=1/2,
 \end{array}
\right.
\end{equation*} 
for some $C > 0$, not dependent on $\eta, \delta$ and $R$, and $\Theta(r)\to 0$ as $r\to 0$.

Now, using Lemma~\ref{est-cone} with $\varphi(t) = t^\beta$ and $\delta_0=\eta_0=\tau$ suitably small depending on $\beta$ (see~\eqref{est-LK-cone}), we obtain
$$
\sup_{K \in \mathcal K_s} \{ \tilde L_K \tilde \varphi(\bar a) \} \leq - \kappa \lambda |\bar{a}|^{\beta-2s} \leq -\kappa_1 \delta^{\frac{2s-\beta}{\beta}} R^{\frac{\gamma}{\beta}(\beta-2s)},
$$
where the last inequality follows from the fact $\delta\varphi(|\bar{a}|)\leq 2 M (R^\gamma + 1)$ (see \eqref{ET1.1G}) and $\kappa, \kappa_1$ are suitable constants. Combining the above estimates,
letting $r\to 0$, we arrive at
$$
\kappa_2 \delta^{\frac{2s}{\beta}} R^{\frac{\gamma}{\beta}(\beta-2s)}\leq \kappa_3 R^{\gamma-\alpha},
$$
for some constants $\kappa_2, \kappa_3$, independent of $R$ and $\delta$, for all large $R$.
Fixing $\beta\in (\gamma, 2s)$ we can choose $\alpha\in (\gamma, 2s)$ to satisfy $2s \gamma<\alpha\beta\Leftrightarrow \gamma(2s-\beta)<\beta(\alpha-\gamma)$. Then the above inequality can not hold
for all large $R$, leading to a contradiction. Hence $u$ must be a constant.

\medskip

Now we consider (ii). We again work with the same coupling function $\Phi$ given by \eqref{AB0004}, and show that $\Phi\leq 0$ in $\Rn\times\Rn$ for all large $R$.
Suppose, to the contrary, that for some $\delta\in (0, 1)$ and large $R$ we have $\sup \Phi> 0$. Then we proceed as before to arrive at \eqref{testing0}. The lower bound for $I$ can be computed from (i) above.
The only difference is the estimate of $I_H$, where we use the (uniform) Lipschitz property of  $H$. In particular,
\begin{align*}
I_H&= -H(\nabla_x \phi(\bar x, \bar y))+H(-\nabla_y \phi(\bar x, \bar y))
\\
&\leq  {\rm H}_{\rm Lip} (|\nabla \tilde \chi_R(\xbar)|+ |\nabla \tilde \chi_R(\ybar)|)\\
& \leq C R^{\gamma-\alpha},
\end{align*}
for some constant $C$, where ${\rm H}_{\rm Lip}$ denotes the Lipschitz constant of $H$. Now the remaining part of the proof can be completed as in (i).
\end{proof}

\bigskip

\begin{proof}[Proof of Theorem~\ref{T1.2}]
In view of Theorem~\ref{T1.1} we only need to consider the case
$\frac{1}{2s} \leq \gamma < \frac{m-2s}{m-1}$. In particular, we have 
$$
m>2s + 1
$$
in this case.

Let  $\varphi$ as in~\eqref{defvarphi} and $\chi$ as in~\eqref{defchi} with $\alpha \in (\gamma, 2s)$. Denote $\tilde \chi_R(x) = M 2^{1+2\alpha} R^{\gamma - \alpha} \chi(x)$, and consider
$$
\Phi(x, y)=u(x)-u(y)-\delta\varphi(|x-y|)- (\tilde \chi_R(x) + \tilde \chi_R(y) ) \quad x, y \in \R^n.
$$
We pick $\epsilon_0 \in (0,1)$, to be chosen later, and set
\begin{equation}\label{choiceR}
R=R(\delta)=\delta^{-\frac{1}{(1-\gamma)\epsilon_0}}\quad \text{for}\; \delta>0,
\end{equation}
The result follows by proving the existence of $\delta_0>0$ such that
$$
\Phi_\delta :=  
\sup_{\Rn\times\Rn}\Phi(x, y) \leq 0
\quad \text{for all}\; \delta\leq \delta_0.
$$


Suppose, to the contrary, that for some sequence $\delta \to 0^+$ we have $\Phi_\delta > 0$.
Since~\eqref{limsup0} holds for the choice of $\tilde \chi_R$,  the supremum of $\Phi_\delta$ is finite and attained at some point $(\bar x, \bar y) \in \R^n \times \R^n$, that is $\Phi(\bar x, \bar y) = \Phi_\delta > 0$. As before, we have $\bar x \neq \bar y$, and $|\bar x|, |\bar y| \leq R/4$. From now on, we denote
$$
\bar a = \bar x - \bar y,
$$
and $\phi(x,y) = \delta \varphi(|x - y|) + \tilde \chi_R(x) + \tilde \chi_R(y)$.

Then, denoting $v_1$ the test function of $u$ with smooth representative $\phi(\cdot, \bar y)$ at $\bar x$, and $v_2$ the test function of $u$ with smooth representative $-\phi(\bar x, \cdot)$ at $\bar y$ , (c.f ~\eqref{eval-u}), we can write
\begin{align*}
-\I v_1(\bar x) + H(\nabla_x \phi(\bar x, \bar y)) \leq 0; \quad -\I v_2(\bar y) + H(-\nabla_y \phi(\bar x, \bar y)) \geq 0.
\end{align*}
We subtract these inequalities to arrive at
\begin{equation*}
I \leq I_H,
\end{equation*}
with
$$
I = -\I v_1(\bar x) + \I v_2(\bar x); \quad I_H = H(-\nabla_y \phi(\bar x, \bar y)) - H(\nabla_x \phi(\bar x, \bar y)).
$$

\medskip

As before, we estimate each term. For $I_H$, we see that
\begin{equation*}
|\nabla \tilde \chi_R(\bar y)|, |\nabla \tilde \chi_R(\bar x)| \leq \kappa M R^{\gamma - 1}
\end{equation*}
for some $\kappa > 0$, due to the fact $\chi_0(|x|/R)=0$ for $|x|\leq R/8$. Hence, from \eqref{HT1.2}, we get
\begin{align*}
I_H \leq & \ C_H \Big{(} |\nabla_x \phi(\bar x, \bar y)|^{m - 1} + |\nabla_y \phi(\bar x, \bar y)|^{m - 1} \Big{)} |\nabla \tilde \chi_R(\bar x) + \nabla \tilde \chi_R(\bar y)|\\
\leq & \ \kappa C_H M R^{\gamma - 1} (\delta^{m-1} + R^{(\gamma-1)(m-1)}),
\end{align*}
from which, by the choice of $R(\delta)$ in~\eqref{choiceR} we arrive at
\begin{equation*}
I_H \leq \kappa C_H M \delta^{m - 1 + \frac{1}{\epsilon_0}},
\end{equation*}
for some $\kappa > 0$, just depending on $n$.

For $I$, using Lemma~\ref{nonlocal-IJ}, we have
\begin{align*}
I \geq & - 2\delta \sup_{K \in \mathcal K_{\rm sym}} \{ \tilde L_K[\cC] \tilde \varphi(\bar a) \} - \sup_{K \in \mathcal K_{\rm sym}} \{ L_K \tilde \chi_R(\bar x) \} - \sup_{K \in \mathcal K_{\rm sym}} \{ L_K \tilde \chi_R(\bar y) \} + \Theta(r)
 \\
=: & -2\delta I_1 - I_2 - I_3 + \Theta(r),
\end{align*}
for some $\Theta(r)\to 0$ as $r\to 0$.
Using Lemma~\ref{est-cone}, with the same choice for $\delta_0, \eta_0$ in~\eqref{choiceetadelta1}, we arrive at the same estimate~\eqref{est-LK-cone}, namely,
\begin{align*}
-2\delta I_1\geq \kappa_1 \delta^{2s-\theta\upbeta} R^{\theta\upbeta(\gamma-1) + \gamma(1-2s) }=\kappa_1 \delta^{2s+\frac{\gamma(2s-1)}{(1-\gamma)\epsilon_0} + \theta\upbeta(\frac{1}{\epsilon_0} - 1)},
\end{align*}
where $\upbeta=(n + 1)/2 + 2 - 2s > 0$ and $\kappa_1$ depends on $n, s, \theta$ but not on $\delta$ and $R$. 

On the other hand, since we are considering symmetric kernels, by Lemma~\ref{est-localization} and the definition of $\tilde \chi_R$ we have
\begin{equation*}
I_2, I_3 \leq C R^{\gamma - 2s} = C \delta^{\frac{2s-\gamma}{(1-\gamma)\epsilon_0}}.
\end{equation*}
This gives a lower bound for $I$. Gathering these estimates and the upper bound of $I_H$, and replacing them into the inequality $I \leq I_H$,
we arrive at
\begin{equation}\label{ET2F}
\kappa_1 \delta^{2s+\frac{\gamma(2s-1)}{(1-\gamma)\epsilon_0} + \theta\upbeta(\frac{1}{\epsilon_0} - 1)} \leq \bar \kappa \left(\delta^{m-1+\frac{1}{\epsilon_0}}
+ \delta^{\frac{2s-\gamma}{(1-\gamma)\epsilon_0}}\right),
\end{equation}
after letting $r\to 0$.
Since $\epsilon_0 < 1$, we have
$$
2s + \frac{\gamma(2s - 1)}{(1 - \gamma) \epsilon_0} < \frac{2s - \gamma}{(1 - \gamma)\epsilon_0}.
$$
On the other hand, since $\gamma \mapsto \gamma(1 - \gamma)^{-1}$ is strictly increasing in $(0,1)$ and $\gamma < \frac{m - 2s}{m - 1}$, we have
\begin{equation*}
1 - \frac{\gamma(2s - 1)}{1 - \gamma} > 2s + 1 - m.
\end{equation*}
Thus, taking $\epsilon_0 \in (0,1)$ close enough to $1$, we get that
\begin{equation*}
2s - m + 1 < \frac{1}{\epsilon_0} - \frac{\gamma(2s - 1)}{(1 - \gamma)\epsilon_0}.
\end{equation*}

Summarizing, if we let 
$$
\tau_1 = 2s + \frac{\gamma(2s - 1)}{(1 - \gamma)\epsilon_0}; \quad \tau_2 = m - 1 + \frac{1}{\epsilon_0}; \quad \tau_3 = \frac{2s - \gamma}{(1 - \gamma)\epsilon_0},
$$
from the choice of $\epsilon_0$ we have $0 < \tau_1 < \tau_2$ and $\tau_1 < \tau_3$, and 
inequality~\eqref{ET2F} reads
\begin{equation*}
\kappa_1 \delta^{\tau_1 + \theta \upbeta(1 - \frac{1}{\epsilon_0})} \leq \bar \kappa (\delta^{\tau_2} + \delta^{\tau_3}).
\end{equation*}
Now fix $\epsilon_0$ and take $\theta$ small enough so that $\tau_1 + \theta \upbeta(1 - \epsilon_0^{-1})< \tau_2, \tau_3$ (this would modify $\kappa_1$ but this is not going to play a role), and then
letting $\delta \to 0^+$ we arrive at a contradiction.
As before, this implies that $u$ is constant.
\end{proof}

Now we come to the proof of Theorem~\ref{T1.3}.
\begin{proof}[Proof of Theorem~\ref{T1.3}] 
In this case, we denote $M = \| u \|_\infty$. Recalling $\varphi$ as in~\eqref{defvarphi}, we consider $\chi$ as in~\eqref{defchi} with $\alpha = 0$, denote $\tilde \chi_R(x) = 2M \chi(x)$, and the function
$$
\Phi(x,y) = u(x)-u(y)-\delta\varphi(|x-y|)- (\tilde \chi_R(x) + \tilde \chi_R(y) ), \ x, y \in \R^n.
$$
We also write 
$$
\phi(x,y) = \delta\varphi(|x-y|) + (\tilde \chi_R(x) + \tilde \chi_R(y) ), \ x, y \in \R^n.
$$
We fix $\varepsilon_0\in (0, (2s)^{-1})$ and let 
$$
R=R(\delta)=\delta^{-1/\varepsilon_0}>\delta^{-1}.
$$ 
Following our general strategy in Section~\ref{Sec-GS}, we claim that there exists $\delta_0 > 0$ such that
$$
\Phi_\delta := \sup_{\Rn\times\Rn}\Phi(x, y)\leq 0
\quad \text{for all}\; \delta\leq \delta_0.
$$

Arguing by contradiction, we assume $\Phi_\delta > 0$. Since~\eqref{limsup0} holds, there exist $\xbar, \ybar \in \R^n, \bar x \neq \bar y$, such that $\Phi(\bar x, \bar y) = \Phi_\delta$. These points are such that $|\bar x|, |\bar y| \leq R/4$. 

Denoting $v_1$ the test function for $u$ with smooth representative $\phi(\cdot, \bar y)$ at $\bar x$, and $v_2$ the test function for $u$ with smooth representative $-\phi(\bar x, \cdot)$ at $\bar y$ , (c.f. ~\eqref{eval-u}), we can write
\begin{align*}
-\I v_1(\bar x) \leq u(\bar x)^q |\nabla_x \phi(\bar x, \bar y)|^m; \quad -\I v_2(\bar y) \geq u(\bar y)^q |-\nabla_y \phi(\bar x, \bar y)|^m.
\end{align*}
Subtract both the inequalities to arrive at
\begin{equation}\label{IIHT1.3}
I \leq I_H,
\end{equation}
with
$$
I = -\I v_1(\bar x) + \I v_2(\bar x); \quad I_H = u(\bar x)^q |\nabla_x \phi(\bar x, \bar y)|^m - u(\bar y)^q |\nabla_y \phi(\bar x, \bar y)|^m.
$$

To compute $I_H$, we first note that
$$
|\grad \tilde \chi_R(\bar x)|, |\grad \tilde \chi_R(\bar y)| \leq C M R^{-1},
$$
for some $C$. Thus, denoting $\bar a = \bar x - \bar y$ and $\hat a = \bar a / |\bar a|$, using \eqref{estvarphi}, we obtain
\begin{align*}
I_H &= u(\xbar)^q|\delta\varphi'(|\bar{a}|)\hat{a}+ \grad \tilde \chi_R(\bar{x})|^m-
u(\ybar)^q|\delta\varphi'(|a|)\hat{a}-\grad \tilde \chi_R(\bar{y})|^m
\\
& = [u(\xbar)^q-u(\ybar)^q] |\delta\varphi'(|\bar{a}|)\hat{a}+ \grad \tilde \chi_R(\bar{x})|^m 
\\
&\qquad + u(\ybar)^q \left(|\delta\varphi'(|\bar{a}|)\hat{a}+ \grad \tilde \chi_R(\bar{x})|^m
- |\delta\varphi'(|\bar{a}|)\hat{a}-\grad \tilde \chi_R(\bar{y})|^m\right)
\\
&\leq \kappa_2 \left[\norm{u}_\infty^q (\delta^m + R^{-m}+ (\delta^{m-1} + R^{-m+1}) R^{-1}\right]
\\
&\leq \kappa_3 \delta^m,
\end{align*}
where $\kappa_3$ depends on $\norm{u}_\infty$, and the last inequality follows from the fact that $R^{-1}<\delta$.

For $I$, using Lemma~\ref{nonlocal-IJ}, we see that
\begin{align*}
I \geq & - 2\delta \sup_{K \in \mathcal K_{s}} \{ \tilde L_K[\cC] \tilde \varphi(\bar a) \} - \sup_{K \in \mathcal K_{s}} \{ L_K \tilde \chi_R(\bar x) \} - \sup_{K \in \mathcal K_{s}} \{ L_K \tilde \chi_R(\bar y) \} + \Theta(r)
 \\
=: & -2\delta I_1 - I_2 - I_3 + \Theta(r),
\end{align*}
for some $\Theta(r)\to 0$ as $r\to 0$, 
and in the same way as in~\eqref{estI1}, we conclude that
\begin{equation*}
I \geq \kappa_4 \delta^{2s-\theta\upbeta} R^{-\theta\upbeta } - CMR^{-1} =\kappa_4 \delta^{2s+\theta\upbeta(\varepsilon_0^{-1}-1)} - CM \delta^{1/\epsilon_0} + \Theta(r).
\end{equation*}
Gathering these estimates in \eqref{IIHT1.3} and letting $r\to 0$, reveals
\begin{equation}\label{ET1.3B}
\kappa_4 \delta^{2s+\theta\upbeta(\varepsilon_0^{-1}-1)}\leq \kappa_5( \delta^m + \delta^{\frac{1}{\varepsilon_0}}).
\end{equation}
Fix $\theta$ small, depending on $\varepsilon_0$, so that
$$
2s+\theta\upbeta(\varepsilon_0^{-1}-1) < \min\{m, 1/\varepsilon_0\}.
$$
With this choice of $\theta$, it is easily seen that  \eqref{ET1.3B} can not hold for all small $\delta$. Hence $\Phi_\delta\leq 0$, and the rest of the proof follows as before.
\end{proof}

Now we come to the proof of Theorem~\ref{T1.4}. This particular proof is different from the previous ones
in the sense that the coercivity structure of the Hamiltonian becomes crucial for our argument.
\begin{proof}[Proof of Theorem~\ref{T1.4}]
We may assume $u>0$. Otherwise, if $u(x_0)=0$ for some $x_0$, we can apply strong maximum principle to obtain $u\equiv 0$, proving the theorem. Consider $\beta, \tilde\gamma$ satisfying
$$
\gamma<\tilde\gamma<\beta< 1.
$$
Later we shall further choose $\tilde\gamma, \beta$ in this range.

For $\mu\in (0, 1)$,  define $\bar{u}=\mu u$. It is direct to see that $\bar u$ solves the problem 
\begin{equation}\label{eqbaru}
\cI \bar{u} - \mu^{1-q-m} \bar{u}^q|\grad\bar{u}|^m=0 \quad \mbox{in} \ \R^n,
\end{equation}
in the viscosity sense.
The proof follows the general strategy (see Section~\ref{Sec-GS}) with subtle modifications. In correspondence with~\eqref{E1.6}, the aim is to prove the following

\medskip
\noindent{\bf Claim A.} \textsl{There exists $\delta _1 \in (0,1)$ such that, for all $\mu \in (1 - \delta_1, 1)$ and $\delta \in (0, \delta_1)$,  there exists 
$\alpha_0>0$, dependent on $\delta, \mu$, such that
$$
\Phi(x, y) :=\mu u(x)-u(y)-\delta|x-y|^\beta-\alpha(1+|x|^2)^{\tilde\gamma/2}\leq 0\quad \text{in}\; \Rn\times\Rn,
$$
for all $\alpha\leq \alpha_0$.}

\smallskip

As before, this concludes the result, since by fixing two points $x, y \in \R^n$, and taking $\alpha \to 0, \mu \to 1$ and $\delta \to 0$
respectively, we arrive at
$
u(x) - u(y) \leq 0,
$
and since the argument is symmetric in $x,y$, we conclude that $u$ is constant in $\R^n$.

Now we prove Claim A. Suppose on the contrary that  for some $\mu < 1$ close to $1$, and $\delta, \alpha > 0$ close to zero, we have
\begin{equation}\label{Phi>0T1.4}
\sup_{x, y}\Phi(x, y)>0.
\end{equation}

We remark that the contradiction argument allows us to take $\alpha$ far smaller than $\delta$ and $1-\mu$.
For the ease of notation, we denote $\psi(x)=(1+|x|^2)^{\tilde\gamma/2}$. Notice that by the growth condition imposed on $u$ we have
$$
\limsup_{|x| \to \infty} \frac{u(x)}{\alpha \psi(x)} < 1, 
$$
and therefore, we can find a number $k$
such that 
$$
|x|\geq k\Rightarrow \Phi(x, y)\leq \mu u(x)- \alpha\psi(x)<0.
$$
On the other hand, since $\gamma<\beta$,
$$\max_{|x|\leq k}\Phi(x, y)\leq \max_{|x|\leq k} (\mu u(x)-\alpha\psi(x)) - u(y) -\delta \min_{|x|\leq k}|x-y|^\beta\to-\infty,$$
as $|y|\to\infty$. Hence, a maximizer $(\xbar,\ybar) \in \R^n \times \R^n$ for~\eqref{Phi>0T1.4} exists.

Denote
$$
\bar{a}=\xbar-\ybar.
$$

Note that since $u>0$ we have $\bar u \leq u$ in $\R^n$. This, together with $\Phi(\xbar,\ybar)>0$ imply
$$
0 < \bar u(\bar x) - u(\bar y) \leq u(\bar x) - u(\bar y),
$$
which gives $\bar a = \bar x - \bar y \neq 0$.

Again, since $\Phi(\xbar,\ybar)>0$, it follows that
$$
\frac{\alpha}{2^{\tilde \gamma/2}} (1 + |\bar x|)^{\tilde \gamma} \leq \alpha\psi(\xbar)\leq \mu u(\xbar)\leq M(1+|\xbar|)^\gamma,
$$
which, in turn, gives
$$
|\bar x|< 1+ |\bar x| \leq C_1 \alpha^{-\frac{1}{\tilde\gamma-\gamma}},
$$
for some $C_1 > 0$ depending on $M$. Moreover, using the same inequality, we also have that
\begin{align}\label{E2.3}
& \delta |\bar{a}|^\beta\leq \bar u(\xbar)\leq \mu M(1 + |\bar x|)^\gamma
\leq M (C_1)^\gamma \alpha^{-\frac{\gamma}{\tilde\gamma-\gamma}}.
 \end{align}
 These bounds will be useful for our calculations below. 

Denote $\varphi(t) = t^\beta$, $\phi(x,y) = \delta \varphi(|x - y|) + \alpha \psi(x)$, and denote $v_1$ the test function for $\bar u$ with smooth representative $\phi(\cdot, \bar y)$ at $\bar x$, and $v_2$ the test function for $u$ with smooth representative $-\phi(\bar x, \cdot)$ at $\bar y$. Using the equations for $u$ and $\bar u$ (recall $\bar u$ solves~\eqref{eqbaru}), and subtracting,  we have
\begin{equation}\label{ET1.4B}
\tilde{I}+\tilde{I}_H\leq 0, 
\end{equation}
where
\begin{align*}
\tilde{I} &=-\I v_1(\bar x) + \I v_2(\bar y), \\
\tilde{I}_H&= \mu^{1-q-m} \bar{u}^q(\xbar)|\delta\beta |\bar{a}|^{\beta-1}\hat{a}+\alpha\grad\psi(\xbar)|^m- {u}^q(\ybar)|\delta\beta |\bar{a}|^{\beta-1}\hat{a}|^m,
\end{align*}
For $\tilde I$, using Lemma~\ref{nonlocal-IJ} we have
\begin{align*}
\tilde I \geq & - 2\delta \sup_{K \in \mathcal K_s} \{ \tilde L_K[\cC] \tilde \varphi(\bar a) \} - \alpha \sup_{K \in \mathcal K_{s}} \{ L_K \psi(\bar x) \}
+ \Theta(r)
  \\
=: & -2\delta I_1 - \alpha I_2 + \Theta(r),
\end{align*}
where $\Theta(r)\to 0$ as $r\to 0$.
A simple modification of Lemma~\ref{est-localization} implies that
\begin{equation}\label{ET1.4C}
I_2 \leq C \Lambda,
\end{equation}
for some $C > 0$, not dependent on $\delta$ and $\alpha$.
More precisely, we note that
$$\sup_{x\in\Rn} [|D^2\psi(x)| + |\grad\psi(x)|]\leq \kappa_2$$
for some $\kappa_2>0$. Thus
$$\sup_{K\in\mathcal{K}} L_K[B_1] (\psi, x, \grad\psi(x))\leq \kappa_3.$$
Again, since $x\mapsto (1+|x|^2)^{\frac{1}{2}}$ is $1$-Lipschitz, we have
$$(1+|x-z|^2)^{\frac{\tilde\gamma}{2}}-(1+|x|^2)^{\frac{\tilde\gamma}{2}}\leq [(1+|x|^2)^{\frac{1}{2}} + |z|]^{\tilde{\gamma}}-(1+|x|^2)^{\frac{\tilde\gamma}{2}}\leq |z|^{\tilde\gamma}.$$
Thus for $s\geq 1/2$, we get
\begin{align*}
\sup_{K\in\mathcal{K}}\int_{|z|\geq 1} ((1+|x-z|^2)^{\frac{\tilde\gamma}{2}}-(1+|x|^2)^{\frac{\tilde\gamma}{2}}) K(z) \dz
\leq \Lambda \int_{|z|\geq 1} |z|^{\tilde\gamma}\frac{1}{|z|^{n+2s}}\dz\leq \kappa_4.
\end{align*}
This gives us \eqref{ET1.4C}.

On the other hand, using Lemma~\ref{est-cone}, we can take $\delta_0, \eta_0$ small enough (not depending on $\bar a$, $\delta$ and $\alpha$) to obtain, for all $K \in \mathcal K_s$, that
\begin{equation*}
\tilde L_K [\cC] \tilde \varphi(\bar a) \leq -c |\bar a|^{\beta - 2} \int_{\cC} |z|^2 K(z)dz \leq -c \lambda |\bar a|^{\beta - 2} \int_{\cC} |z|^{2 - n - 2s}dz,
\end{equation*}
for some $c > 0$. Using the estimates on the cone in~\cite[Example 1]{Barles-12}, we arrive at
\begin{equation*}
I_1 \leq -c \lambda |\bar a|^{\beta - 2s}. 
\end{equation*}
Combining the above estimates, we conclude that
\begin{equation}\label{E2.5}
\tilde I \geq \kappa_1 \delta |\bar{a}|^{\beta-2s} - C\Lambda \alpha + \Theta(r),
\end{equation}
where the constants $\kappa_1, C$ do not depend on $\alpha, \delta, \mu$.

\medskip

Now we compute the term $\tilde{I}_H$. We  write
$$
\tilde{I}_H= (\mu^{1-q-m} -1)  \underbrace{\bar{u}^q(\xbar)|\delta\beta|\bar a|^{\beta-1}\hat{a}+ \alpha\grad\psi(\xbar)|^m}_{\Xi_1} +
\underbrace{\Big{(} \bar{u}^q(\xbar) |\delta|\beta\bar a|^{\beta-1}\hat{a}+ \alpha\grad\psi(\xbar)|^m - u(\ybar)\delta^m\beta^m|\bar a|^{m(\beta-1)} \Big{)}}_{\Xi_2}.
$$
Choose $\beta$ close to $1$ so that
\begin{equation}\label{E2.6}
\max\left\{\frac{(1-\beta)\gamma}{(\tilde\gamma-\gamma)\beta}, \frac{m(1-\beta)\gamma}{(\tilde\gamma-\gamma)\beta}\right\} < 1.
\end{equation}
Using~\eqref{E2.3}, we have 
\begin{equation}\label{E2bis}
|\bar{a}|^{1-\beta}\leq \kappa_6\delta^{\frac{\beta-1}{\beta}}\alpha^{-\frac{\gamma(1-\beta)}{(\tilde\gamma-\gamma)\beta}}
\ \Rightarrow \ \alpha |\bar a|^{1-\beta} \leq \kappa_6 \delta^{\frac{\beta-1}{\beta}}\alpha^{1-\frac{\gamma(1-\beta)}{(\tilde\gamma-\gamma)\beta}},
\end{equation}
for some $\kappa_6 > 0$, not dependent on $\delta$ and $\alpha$. By the choice of $\beta$ in~\eqref{E2.6}, 
we have the exponent  of $\alpha$ positive.
Thus, since $\nabla \psi(\bar x)$ is bounded with an upper bound independent of $\delta$ and $\alpha$, for  $\delta > 0$ we can find $\alpha$ small enough so that
\begin{align*}
\Xi_1&= \bar{u}(\xbar)^q (\delta\beta |\bar{a}|^{\beta-1})^m \left|\hat{a}+  \delta^{-1}\beta^{-1} \alpha |\bar{a}|^{1-\beta} \grad\psi(\xbar) \right|^m
\\
&\geq \frac{1}{2}\bar{u}(\xbar)^q \delta^m\beta^m |\bar{a}|^{m(\beta-1)},
\end{align*}
where we use the fact that $m>0$ and $|\hat{a}|=1$.
Similarly, for $\Xi_2$, using that  $\bar{u}(\xbar)>u(\ybar)$ and $m > 0$, and taking $\alpha$ small in terms of $\delta$, we see that
\begin{align*}
\Xi_2 &\geq u(\ybar)^q (\delta\beta |\bar{a}|^{\beta-1})^m\left(|\hat{a}+  \delta^{-1}\beta^{-1} \alpha |\bar{a}|^{1-\beta} \grad\psi(\bar x)|^m-1\right) \\
& \geq - \kappa u(\ybar)^p \alpha \delta^{m - 1}\beta^{m-1} |\bar{a}|^{(\beta-1)(m - 1)},
\end{align*}
for some constant $\kappa > 0$ depending on $m$ and $\norm{\grad \psi}_\infty$, but not on $\delta$ and $\alpha$, where the 
last inequality follows from the Lipschitz property of the map $z\mapsto |\hat{a}+z|^m$ in $\{|z|\leq 1/2\}$.

Thus, using the above estimates for $\Xi_1, \Xi_2$ and that $\bar u(\bar x) > u(\bar y)$, we see that
\begin{align*}
\tilde{I}_H &\geq \bar u(\bar x)^q \delta^m\beta^m |\bar a|^{m(\beta - 1)} \Big{(} \frac{1}{2}(\mu^{1-q-m} -1) - \kappa  \delta^{-1}\beta^{-1} \alpha |\bar a|^{1 - \beta} \Big{)}
\\
&\geq \frac{(\mu^{1-q-m} -1)}{4} \bar{u}(\xbar)^q \delta^m\beta^m |\bar{a}|^{m(\beta-1)}
\end{align*}
for $\alpha$ small enough dependent on $\mu$ and $\delta$.

Combining the above estimate with \eqref{E2.5} in \eqref{ET1.4B}, letting $r\to 0$, we see that for all $\alpha$ small enough in terms of $\delta$ and 
$\mu$, we have
\begin{equation}\label{E2.7}
 \kappa_1 \delta |\bar{a}|^{\beta-2s} + \frac{(\mu^{1-q-m} -1)}{4} \bar{u}(\xbar)^q \delta^m |\bar{a}|^{m(\beta-1)}  \leq C_1 \Lambda \alpha.
\end{equation}

Now we can have one of the following two situations.

\medskip

\noindent
{\bf Case 1.} There exists a subsequence $\alpha_k\to 0$ and a constant $\kappa_9$ such that $|\bar{a}|=|\bar{a}(\alpha_k)|\leq \kappa_9$ for all $k$. Since $\beta<2s$, the left most term in
\eqref{E2.7} remains positive, whereas the r.h.s. goes to zero. This is a contradiction.

\smallskip

\noindent
{\bf Case 2.} There exists a subsequence $\alpha_k\to 0$ such that $|\bar{a}_k|=|\xbar_k-\ybar_k|\to \infty$ as $k\to\infty$. From $\Phi(\xbar_k,\ybar_k)>0$ we see that
$$
\delta |\bar a_k|^\beta\leq \bar{u}(\bar x_k)\;\Rightarrow \bar{u}(\bar x_k)^q \geq \delta^q |\bar a_k|^{q\beta}>\kappa>0,
$$
for some $\kappa>0$, as $k\to\infty$. Therefore, for all large $k$, dropping the first term in the l.h.s. of \eqref{E2.7}, and using~\eqref{E2.3} (more speciffically~\eqref{E2bis}), we see that
$$
\kappa\frac{(\mu^{1-q-m} -1)}{4}\delta^m \leq \kappa_9 \alpha_k |\bar a_n|^{m(1-\beta)}\leq \kappa_{10} \alpha_k^{1-\frac{m(1-\beta)\gamma}{\beta(\tilde\gamma-\gamma)}},
$$
where the constant $\kappa_{10}$ does not depend on $\alpha_k$. By \eqref{E2.6}, the exponent of $\alpha_k$ is positive, and therefore, the inequality cannot hold for large $k$.
Therefore, \eqref{E2.7} cannot hold for arbitrary small $\alpha$, proving Claim A.
\end{proof}

Next, we prove Theorem~\ref{T1.5}

\begin{proof}[Proof of Theorem~\ref{T1.5}]
Suppose, on the contrary, that $\gamma> \max\{0, \frac{m-2s}{m-1}\}$.
Since $u+c$ also solves \eqref{ET1.5A}, we assume that $u(0)=0$. Let
$$ A(r):=\max_{\bar{B}_r(0)}|u(x)|.$$
In view of \eqref{ET1.5B}-\eqref{ET1.5C} we have
$$0< \limsup_{r\to\infty} \frac{A(r)}{r^\gamma}<\infty.$$
Let $r_k\to\infty$ be such that
$$\frac{1}{C_1} r^\gamma_k\leq A(r_k)\leq C_1 r^\gamma_k\quad \text{for all}\; n\geq 1,$$
for some large constant $C_1$. Let us define
$$v_k(0)=0,\quad v_k(x)=\frac{1}{A(r_k)} u(r_kx)\quad x\in \Rn.$$
Then, from the above estimate and \eqref{ET1.5B} it follows that
$$\max_{\bar{B}}|v_k|=1\quad \text{and}\quad |v_k(x)|\leq C(1+|x|)^\gamma,$$
for all $k\geq 1$ and some constant $C$, independent of $n$. For each $k\geq 1$, let us now define a nonlocal
operator 
$$[\cI_k \phi (\cdot)](x)=[r^{2s}_k \cI \phi(r^{-1}_k\cdot)](r_k x).$$
Then $\cI_k$ is elliptic with respect to the family $\mathcal{K}_{\rm sym}$.
That is, 
$$\inf_{K\in \mathcal{K}_{\rm sym}} (L_K u(x)-L_Kv(x))\leq \cI_n u(x)-\cI_k v(x) \leq \sup_{K\in\mathcal{K}_{\rm sym}} (L_K u(x)- L_K v(x)),$$
for all $u, v\in C^{1,1}(x)\cap L^1(\omega_s)$. From \eqref{ET1.5A} we see that
\begin{equation}\label{ET1.5D}
\cI^k v_k - \left[{A(r_k)}{r_k^{-\frac{m-2s}{m-1}}}\right]^{m-1}|\grad v_k|^m=0\quad \text{in}\; \Rn.
\end{equation}

\noindent{\bf Claim:} The family $\{v_k\}$ locally equicontinuous in $\Rn$.

Let us first complete the proof assuming the claim. Since $\{v_k\}$ locally equicontinuous, we can extract a subsequence $v_{k_\ell}\to v$, uniformly on compacts, where $v\in C(\Rn)$ satisfying
$$v(0)=0\quad \max_{\bar{B}}|v|=1\quad \text{and}\quad |v(x)|\leq C(1+|x|)^\gamma.$$
Moreover, since
$${A(r_k)}{r_k^{-\frac{m-2s}{m-1}}}\to\infty\quad \text{as} \; n\to\infty,$$
using the stability property of viscosity solution and \eqref{ET1.5D}, we have
$$-|\grad v|=0\quad \text{in}\; \Rn$$
in the viscosity sense. That is, for any smooth test function $\varphi$ touching $v$ from above at a point $x$ we must have $|\grad \varphi(x)|=0$. It is easily seen that $v$ must be constant, and therefore, 
$v\equiv 0$. This contradicts the fact that $\max_{\bar{B}}|v|=1$. Hence $\gamma\leq \max\{0, \frac{m-2s}{m-1}\}$.

{\it Proof of the claim.}  The proof follows from the argument of \cite[Lemma~3.6]{BT23} and we provide only a sketch here for convenience. Fix $\ell>0$ and define
$$\Theta= 2\sup_k \, \sup_{B_{\ell+1}} |v_k|<\infty.$$
Let $\psi:\Rn\to[0, 1]$ be a smooth cut-off function satisfying $\psi=0$ in $B_\ell$ and $\psi=1$ in $B^c_{\ell +\frac{1}{2}}$. To prove the claim, it is enough to show that for each $\eta\in (0,1)$ and 
$\varepsilon\in (0, 1)$, there exists $L>0$ satisfying
\begin{equation}\label{ET1.5E}
v_k(x)-v_k(y)\leq L|x-y|^\eta + 2\Theta(1-\varepsilon) + \Theta\psi(x) \quad \text{for all}\; x, y\in B_{\ell+1},\; n\geq 1.
\end{equation}
To show \eqref{ET1.5E} we consider the function
$$\Phi(x, y)=\varepsilon v_k(x)-v_k(y)- L|x-y|^\eta - \Theta\psi(x),$$
and assume on the contrary that 
$$\sup_{B_{\ell+1}\times B_{\ell+1}} \Phi>0,$$
for all large $L$ satisfying $L >2^{1+2\eta}\Theta$. Let $(\xbar,\ybar)\in B_{\ell+1}\times B_{\ell+1}$ be a point where the supremum is attained. Also, we only need to analyze the case where $\xbar\neq\ybar$.
Since $\Phi(\xbar, \ybar)\geq \Phi(\xbar,\xbar)$ we obtain $L|\xbar-\ybar|^\eta\leq \Theta$, implying $|\xbar-\ybar|<1/4$. Also, let
$$p=L\eta |\xbar-\ybar|^{\eta-2}(\xbar-\ybar),\quad q=\grad\psi(\xbar).$$
It is important for us to observe here
$$|p|=L\eta |\xbar-\ybar|^{\eta-1}\geq \eta\Theta^{\frac{\eta-1}{\eta}} L^{1/\eta}\to \infty,$$
as $L\to\infty$. Denoting
$$\varrho_k = \left[{A(r_k)}{r_k^{-\frac{m-2s}{m-1}}}\right]^{m-1},$$
we observe from \eqref{ET1.5D} that
$$\cI^k (\varepsilon v_k) -\varrho_k \varepsilon^{1-m} |\grad(\varepsilon v_k)|^m=0\quad \text{in}\; \Rn.$$
 Now applying sub and super-solution inequality and then subtracting the inequalities we obtain the estimate
 \begin{align}\label{ET1.5F}
 \varrho_k \varepsilon^{1-m} |p+\Theta q|^m -\varrho_k  |p|^m & \leq \underbrace{\sup_{K\in\mathcal{K}_{\rm sym}}\{L_K[B_\delta] (\phi(\cdot, \ybar), \xbar, p+\Theta q)-L_K[B_\delta] (\phi(\xbar,\cdot), \ybar, p)\}}_{:=A_1}\nonumber
 \\
 &\quad + \sup_{K\in\mathcal{K}_{\rm sym}}\{L_K[B^c_\delta] (\varepsilon v_k, \xbar, p+\Theta q)-L_K[B^c_\delta] (v_k, \ybar, p)\},
 \end{align}
 for all $\delta>0$ small, where $\phi(x, y):=L|x-y|^\eta + \Theta \psi(x)$ and
 $$L_K[D] (f, x, \xi)=\int_D (f(x+z)-f(x)-1_{B_{1}}(z) z\cdot\xi) K(z)dz.$$
 Using the maximality of $\Phi$ at $(\xbar, \ybar)$, it is standard to show that 
 $$\sup_{K\in\mathcal{K}_{\rm sym}}\{L_K[B^c_\delta\cap B_{\frac{1}{4}}] (\varepsilon v_n, \xbar, p+\Theta q)-L_K[B^c_\delta\cap B_{\frac{1}{4}}] (v_n, \ybar, p)\}\leq \kappa \Theta,$$
 for some constant $\kappa$. Since $z\mapsto |z|$ is integrable with respect to $\omega_s$, it follows that 
  $$\sup_{K\in\mathcal{K}_{\rm sym}}\{L_K[ B^c_{\frac{1}{4}}] (\varepsilon v_k, \xbar, p+\Theta q)-L_K[B^c_{\frac{1}{4}}] (v_k, \ybar, p)\}\leq \kappa_1(1+ \Theta+|p|),$$
  for some constant $\kappa_1$. Since $A_1\to 0$ as $\delta\to 0$, the rhs of \eqref{ET1.5F} is dominated by $\kappa \Theta + \kappa_1(1+ \Theta+|p|)$.
 Note that for some constants $\kappa_2, \kappa_3$ we have 
 \begin{align*}
\varepsilon^{1-m} |p+\Theta q|^m - |p|^m &\geq (1-\varepsilon)(m-1)|p+\Theta q|^m -\kappa_2\Theta |q| (|p|^{m-1}+|\theta q|^{m-1})
\\
&\geq \frac{1}{2}(1-\varepsilon)(m-1)|p|^m -\kappa_3
 \end{align*}
Combining these estimates in \eqref{ET1.5F} we get
$$\frac{\varrho_k}{2}(1-\varepsilon)(m-1)|p|^m -\kappa_3\leq \kappa \Theta + \kappa_1(1+ \Theta+|p|).$$
Since $|p|\to\infty$ as $L\to\infty$ and $m>1$, we obtain a contradiction from the above estimate. This proves the claim.
\end{proof}

We finish this section by proving Theorem~\ref{T1.6}

\begin{proof}[Proof of Theorem~\ref{T1.6}]
We show that if there is a solution $u$ to \eqref{ET1.6A} then it must be positive constant. It would lead to a contradiction, since $f$ does not vanish on $(0, \infty)$.
As before (see Theorem~\ref{T1.1}(i)),  we consider 
$$
\Phi(x,y) := u(x)-u(y)-\delta\varphi(|x-y|)- M 2^{1+2\alpha} R^{\gamma - \alpha} (\chi(x) + \chi(y) ).
$$
where $\varphi(t)=t^\beta$ for some $\beta\in (\gamma, 2s\wedge 1)$. In this case, \eqref{testing0} will be replaced by
\begin{equation}\label{ET1.6C}
I\leq f(u(\xbar))-f(u(\ybar)),
\end{equation}
where 
$$I :=  -\cI v_1(\xbar) + \cI v_2(\ybar),$$
and $v_1, v_2$ are defined as in Theorem~\ref{T1.1}.
Since 
$$u(\xbar)-u(\ybar)\geq \delta\varphi(|\xbar-\ybar|)>0,$$
we have $f(u(\xbar))-f(u(\ybar))\leq 0$. In view of Lemma~\ref{est-localization} we obtain (see Theorem~\ref{T1.1}(i))
\begin{equation*}
I \geq
 \kappa_1 \delta^{\frac{2s}{\beta}} R^{\frac{\gamma}{\beta}(\beta-2s)} + \Theta(r) - C R^{\gamma - 2s},
\end{equation*}
where $\Theta(r)\to 0$ as $r\to 0$.  Letting $r\to 0$ we get
\begin{equation}\label{E3.30}
\kappa_1 \delta^{\frac{2s}{\beta}} R^{\frac{\gamma}{\beta}(\beta-2s)} \leq C R^{\gamma - 2s},
\end{equation}
where the constants $C, \kappa_1$ do not depend on $R$. Since $\gamma(2s-\beta)<\beta(2s-\gamma)$, \eqref{E3.30} can not hold for 
large $R$. Thus $\Phi\leq 0$ for all large $R$, proving $u$ to be a constant. This completes the proof of (i).

For the second part, we fix $\beta$ as given by \eqref{Cnd-1}.
Now  we estimate $f(u(\xbar))-f(u(\ybar))$ as follows: using the convexity of $f$ we get
\begin{align*}
f(u(\xbar))-f(u(\ybar))\leq - f'(u(\xbar)) (u(\ybar)-u(\xbar)).
\end{align*}
Since $u(\xbar)-u(\ybar)\geq \delta\varphi(|\xbar-\ybar|)=\delta |a|^\beta$ , we have
$$f(u(\xbar))-f(u(\ybar))\leq  f'(u(\xbar)) \delta |a|^\beta,$$
using the fact $f'\leq 0$.
Since $f'$ is increasing, using the fact that $u(\xbar)\leq C R^\gamma$, \eqref{ET1.6C} would give us
$$ \kappa_1 |a|^{\beta-2s} + \delta |a|^\beta |f'(C R^\gamma)|=\kappa_1 |a|^{2s-\beta} - \delta |a|^\beta f'(C R^\gamma)\leq C_1 R^{\gamma-2s}.$$
Applying Young's inequality with $p=2s$ and $p'=\frac{2s}{2s-1}$, to the left-hand side, we note that
$$\kappa_2 |a|^{\frac{\beta-2s}{2s}+ \frac{\beta(2s-1)}{2s}} |f'(C R^\gamma)|^{\frac{2s-1}{2s}}\leq C_1 R^{\gamma-2s},$$
for some constant $\kappa_2$. Since $\delta |a|^\beta\leq C R^\gamma $, we get 
$|a|^{\beta-1}\geq \kappa_3 R^{(\beta-1)\gamma}$, implying
$$ R^{2s-(2-\beta)\gamma}|f'(C R^\gamma)|^{\frac{2s-1}{2s}}\leq C_2$$
for some constant $C_2$, independent of $R$. This contradicts \eqref{Cnd-1} for a sequence of $R$ tending to infinity. The rest of the proof
follows as before.
\end{proof}

\section{Applications of the Liouville property to regularity}\label{S-appl}
In this section, we showcase application of above Liouville properties in regularity estimates. 

We start with interior regularity estimates for HJ equations with \textsl{critical} diffusion, in the sense that the nonlocal operator has a diffusion comparable to the one of the square root of the fractional Laplacian $\sqrt{-\Delta}$, that is, when $s=1/2$. 

In this setting, we can mention the contributions by Silvestre~\cite{Sil12b,Sil12a}, where the author considers operators of the form $\sqrt{-\Delta} u + b(x)\cdot \grad u=f$. Assuming 
$b, f\in L^\infty$, it is shown in \cite{Sil12a} that $u\in C^\alpha$ for some $\alpha\in (0, 1)$ whereas \cite{Sil12b} establishes $C^{1,\alpha}$
regularity of the solution assuming $C^\alpha$ regularity of $b$ and $f$. Later, Schwab and Silvestre \cite{SS16} prove $C^\alpha$ regularity 
of solutions of \eqref{ET3.1A} for some $\alpha$, dependent on $\lambda, \Lambda, C_2$ and $n$. In \cite[Theorem~3.1]{CGT}, the authors consider Pucci type nonlinear equation with H\"{o}lder continuous coefficients and obtain Lipschitz regularity of the solutions. It should be 
noted that \cite{CGT} deals with non-symmetric kernels. Theorem~\ref{T3.1} improves \cite{Sil12a,Sil12b,SS16} by proving $C^\gamma$ regularity 
for any $\gamma\in (0, 1)$ and complements the regularity result of \cite{CGT} as we do not impose any regularity hypothesis on 
the coefficients other than continuity.

To state our next result, we introduce the following state dependent operator corresponding to
$s=\frac{1}{2}$.
$$\mathscr{J} u(x)=\inf_{a\in\mathcal{A}}\,\sup_{b\in\mathcal{B}}\int_{\Rd} (u(x+z)-u(x)-1_B(z)z\cdot \grad u(x))\frac{k_{ab}(x, z)}{|z|^{n+1}}dz,$$
where $\mathcal{A}, \mathcal{B}$ are some index sets. We impose the following condition on the kernels $k_{ab}$:

\medskip

\noindent
{\bf (K)} \textsl{The family $k_{ab}$ is symmetric in the sense that $k_{ab}(z)=k_{ab}(-z)$.
$x\mapsto k_{ab}(x, z)$ continuous in $B_1$, uniformly with respect to $a, b, z$. That is, for any compact set $\mathcal{K}\subset B_1$,
there exists a modulus of continuity function $\varrho$ satisfying}
$$
|k_{ab}(x_1, z)-k_{ab}(x_2, z)|\leq \varrho(|x_1-x_2|)\quad \text{for all}\; x_1, x_2\in\mathcal{K}, \; a\in\mathcal{A},\, b\in\mathcal{B},\; z\in\Rn.
$$

\textsl{Furthermore, for some constants $0<\lambda\leq \Lambda$ we have $\lambda\leq k_{ab}\leq \Lambda$ for all $(a, b)\in\mathcal{A}\times\mathcal{B}$.}

\medskip

The next theorem is a consequence of~\ref{T1.1}. 
\begin{thm}\label{T3.1}
 Consider the operator $\mathscr{J}$ satisfying {\bf (K)}.
 Suppose that $H:\Rn\times\Rn\to \R$  be a continuous function satisfying the following:
 \begin{itemize}
\item[(i)] There exist constants $C_1, C_2$ satisfying
$$|H(x, p)|\leq C_1 + C_2 |p|\quad \text{for all}\; (x, p)\in B_1\times \Rn;$$
\item[(ii)] There exists a $\breve{H}$, Lipschitz in $p$, such that 
$$ \lim_{t\to 0+} t H(x, \frac{1}{t}p)=\breve{H}(x, p),$$
uniformly on compact subsets of $\Rn\times\Rn$.
 \end{itemize}
Then for any $\gamma\in (0, 1)$,  the viscosity solution $u$ to
\begin{equation}\label{ET3.1A}
 -\mathscr{J}u + H(x, u)= f\quad \text{in}\; B_1
\end{equation}
is in $C^\gamma(B_{\frac{1}{2}})$, and the $C^\gamma$ norm depends only on $\lambda, \Lambda, n, \gamma, \sup_{B_1}\abs{u}, \norm{u}_{L^1(\omega_{\nicefrac{1}{2}})}, \norm{f}_\infty$ and
modulus of continuity $\varrho$.
\end{thm}

\begin{proof}
We follow the idea of Serra \cite{Ser15} (see also \cite[section~5]{BK23}). Firstly, consider a smooth cut-off function $\xi:\Rn\to [0, 1]$
such that $\xi(y)=1$ for $|y|\leq \frac{7}{8}$ and $\xi(y)=0$ for $|y|\geq 1$. Letting $v=\xi u$ we see that
\begin{equation}\label{ET3.1AA}
f - C\Lambda \norm{u}_{L^1(\omega_s)}\leq  -\mathscr{J}v + H(x, v)\leq f + C\Lambda \norm{u}_{L^1(\omega_s)} \quad \text{in}\; B_{\frac{3}{4}},
\end{equation}
for some constant $C$.
Dividing both side by $\tau:=1+\norm{u}_{L^\infty(B_1)}+   \norm{u}_{L^1(\omega_{\nicefrac{1}{2}})} + \norm{f}_{L^\infty(B_1)}>1$, and replacing 
$v$ by $\frac{v}{\tau}$ and $H$ by
$\frac{1}{\tau}H(x, \tau p)$
(which also satisfy the above condition with the same $C_1, C_2$)
we may assume that $\norm{u}_{L^\infty(B_1)}+   \norm{u}_{L^1(\omega_{\nicefrac{1}{2}})} + \norm{f}_{L^\infty(B_1)}\leq 1$.

Set $\gamma\in (0, 1)$. We prove the theorem by contradiction. Suppose that there exists a family of nonempty index sets $\mathcal A_k, \mathcal B_k$ (with kernels satisfying the continuity assumption {\bf (K)} uniformly in $k$), and a sequence of $u_k, f_k$ solving
$$-\mathscr{J}u_k + H(x, \grad u_k)=f_k \quad \text{in}\; B_1$$
such that
$$
\lim_{k\to\infty} \sup_{x,y\in B_{\frac{1}{2}}}\frac{|u_k(x)-u_k(y)|}{|x-y|^\gamma}=\infty.
$$
Let $v_k=\xi u_k$, where $\xi$ is defined above, and we consider the equation \eqref{ET3.1AA} satisfied by $v_k$.
As in~\cite[Lemma~ 4.3]{Ser15}, the above is equivalent to
\begin{equation}\label{T1.2A}
    \sup_k \, \sup_{r>0}\, \sup_{z\in B_{\frac{1}{2}}}\,
r^{-\gamma} \norm{v_k(\cdot)-v_k(z)}_{L^\infty(B_r(z))}=+\infty.    
\end{equation}

For $r > 0$, define 
$$
\Theta(r)=  \sup_k \, \sup_{r_1\geq r}\, \sup_{z\in B_{\frac{1}{2}}}\,
r_1^{-\gamma} \norm{v_k-v_k(z)}_{L^\infty(B_{r_1}(z))},
$$
which is nondecreasing as $r \to 0$, and $\Theta(r) \to +\infty$ as $r \to 0$.

Now choose a sequence $(z_m, k_m, r_m) \in B_{1/2} \times \mathbb{N} \times (0,1)$ with $r_m\to 0$ such that
$$r_m^{-\gamma} \norm{v_{k_m}-v_{k_m}(z_m)}_{L^\infty(B_{r_m}(z_m))}
> \frac{\Theta(r_m)}{2},
$$
and define
$$\tilde{v}_m(y)=\frac{v_{k_m}(z_m + r_m y)-v_{k_m}(z_m)}{r_m^{\gamma}\Theta(r_m)}.
$$

Then, the following can be easily verified 
\begin{equation}\label{T1.2B}
 \tilde{v}_m(0)=0, \quad \norm{\tilde{v}_m}_{L^\infty(B_1)}\geq 1/2, \quad
\quad |\tilde{v}_m(y)|\leq \kappa (1+|y|)^\gamma, \quad y \in \R^n,
\end{equation}
for some constant $\kappa$, independent of $m$. The first two facts are direct from the definition and the choice of $r_m$. For the last inequality, by the monotonicity property of 
$\Theta$ it follows that
$$
\norm{\tilde{v}_m}_{L^\infty(B_R)}\leq R^\gamma \frac{\Theta(Rr_m)}{\Theta(r_m)}\leq R^\gamma.
$$
Moreover, defining
$$
-\mathscr{J}_m \phi=\inf_{a\in\mathcal{A}_{k_m}}\sup_{b\in\mathcal{B}_{k_m}}\int_{\Rn}(\phi(x+z)-\phi(x)-1_{B}(z) z\cdot\grad \phi(x))\frac{k_{ab}(z_m+r_mx, r_m z)}{|z|^{n+1}}\dz,
$$
we see from \eqref{ET3.1AA} that
$$\left|- \mathscr{J}_m v_m
 + \frac{r^{1-\gamma}_m}{\Theta(r_m)} H(z_m + r_m x, r^{\gamma -1}_m\Theta(r_m)\grad v_m)\right|\leq  \frac{r^{1-\gamma}_m}{\Theta(r_m)} |f(z_m+r_m x)| + C\Lambda \frac{r^{1-\gamma}_m}{\Theta(r_m)}
 \quad \text{in}\; B_{\frac{1}{2r_m}}.$$
Due to compactness, we shall assume without any loss of generality that $z_m\to z_\circ\in \bar{B}_{\frac{1}{2}}$, as $m\to\infty$.
One can also extract a subsequence $m_k$ such that
$$\mathscr{J}_{m_k}\to \tilde{\mathscr{J}}$$
weakly (see \cite[Lemma~5.4]{BK23}), as $m_k\to\infty$, where $\tilde{\mathscr{J}}$ is a translation invariant, positively $1$-homogeneous operator which is elliptic with respect to $\mathcal{K}_{\rm sym}$.
Since 
$$\left|\frac{r^{1-\gamma}_m}{\Theta(r_m)} H(z_m + r_m x, r^{\gamma -1}_m\Theta(r_m)\grad v_m)\right|
\leq C_1 + C_2 |\grad v_m|,
$$
applying \cite[Theorem~7.2]{SS16} we see that $\{\tilde{v}_m\}$ is locally H\"{o}lder continuous, uniformly in $m$. Thus, 
we can extract a subsequence of $\tilde{v}_k$ converging to $w$ so that 
$$-\tilde{\mathscr{J}} w + \breve{H}(z_\circ, \grad w) =0\quad \text{in}\; \Rn,$$
and, due to \eqref{T1.2B}, we also have
\begin{equation}\label{T1.2C}
 w(0)=0, \quad \norm{w}_{L^\infty(B_1)}\geq 1/2, \quad
\quad |w(y)|\leq \kappa (1+|y|)^\gamma.
\end{equation}
 Applying Theorem~\ref{T1.1}, we see that
$w$ is a constant, but this contradicts the first two criteria in \eqref{T1.2C}. Hence \eqref{T1.2A} can not hold.
This completes the proof.
\end{proof}

Our next result is an application of Theorem~\ref{T1.6}. To understand the result, let us consider the following equation
$$-\cI_{\rm sym} = u^{-\delta}\quad \text{in}\; \Omega, \quad u>0\quad \text{in}\; \Omega,\quad \text{and} \quad u=0\quad \text{in}\; \Omega^c,$$
in a smooth bounded domain $\Omega$. Besides the existence-uniqueness of the solutions, the regularity of the solutions up to the boundary
has been of special interest. For the classical case of Laplacian we refer to \cite{Gui-Lin93,LM91}. A similar problem for the nonlocal operator 
is considered recently in \cite{AGS18,BBMP} for $-\cI_{\rm sym}=(-\Delta)^s$. More precisely, it is shown in \cite{AGS18} that 
$u\in C^{s}(\bar\Omega)$ for $\delta<1$ and $u\in C^{\frac{2s}{\delta+1}}(\bar\Omega)$ for $\delta>1$. The proof techniques rely on the 
Green function representation of the solutions and sharp boundary regularity of the Green function. As can be seen from the proofs of 
\cite{AGS18,Gui-Lin93}, the boundary regularity of $u$ is {\it dictated} by the boundary regularity of the torsion function or the principal eigenfunction. 
Interestingly, for the operator $\cI_{\rm sym}$ one may not expect a sharp boundary behaviour of the Dirichlet solution as point out by
Ros-Oton and Serra in \cite{RS-Duke}. Our next result provides some insight on the regularity of the solution, compare with the "torsion problem" when $\delta$ is large.

\begin{thm}
Let $u$ be a positive viscosity solution to 
$$-\cI_{\rm sym} u + 1_{(1, 2)}(s)H(\grad u) = \frac{k(x)}{u^\delta} \quad \text{in}\; \Omega, \quad u>0\quad \text{in}\; \Omega,\quad\text{and}\quad
u=0\quad \text{in}\; \Omega^c,$$
for some $\delta>0$ where $k\in C(\bar{\Omega})$ and $k>0$ in $\bar\Omega$. $H$ is positively 1-homogeneous and locally Lipschitz. 
Let $\psi$ be the solution to the torsion problem
$$-\cI_{\rm sym} \psi + 1_{(1, 2)}(s) H(\grad \psi)= 1\quad \text{in}\; \Omega,\quad \psi>0\quad \text{in}\; \Omega,\quad \text{and}
 \quad u=0\quad \text{in}\; \Omega^c,$$
and for some $0<\beta<1\wedge 2s$ we have
$$(\dist(x, \Omega^c))^\beta\leq C\psi(x) \quad \text{in}\; \Omega,$$
for some constant $C$.
If $\beta(1+\delta)>2s$, then we have
$$(\dist(x, \Omega^c))^\beta\leq C_1 u(x),\quad\text{and}\quad \sup_{x\in\Omega}\,\frac{u(x)}{(\dist(x, \Omega^c))^\beta}=\infty.$$
In particular,  $u\notin C^\beta(\bar\Omega)$.
\end{thm}

\begin{proof}
To simplify the notation, we write $1_{(1, 2)}(s) H(p)= H_s(p)$.
The first part follows from the comparison principle. More precisely, let  $k\geq c_k$ in $\bar\Omega$, for some positive constant $c_k$. Then
$$ -\cI_{\rm sym} u + H_s(\grad u) \geq  \frac{c_k}{\norm{u}^\delta_\infty}\quad \text{in}\; \Omega.$$
By comparison principle \cite{BCI08}, we have 
$$\frac{c_k}{\norm{u}^\delta_\infty} \psi\leq u\quad \text{in}\; \Rd.$$
This gives the first part. Now we consider the second part. Suppose, on the contrary, that
\begin{equation}\label{E1.9}
   \sup_{x\in\Omega}\,\frac{u(x)}{(\dist(x, \Omega^c))^\beta}=\kappa\in (0, \infty). 
\end{equation}
In particular, we have from \eqref{E1.9} that
$$u(x)\leq \kappa (\dist(x, \Omega^c))^\beta.$$
Consider a sequence of points $\{x_n\}$ in $\Omega$ approaching the boundary. Let $r_n=\dist(x_n, \Omega^c)$.
Set $\gamma=\frac{\beta(1+\delta)}{2s}\in (1, \infty)$ and define
$$w_n(y)=\frac{1}{r^\beta_n} u(x_n + r^\gamma_n y),\quad y\in\Rd.$$
We claim that
\begin{equation}\label{E1.10}
|w_n(y)|\leq \kappa_1 (1+|y|^\beta)\quad \text{for}\; y\in \Rd,
\end{equation}
for some $\kappa_1$, independent of $n$. Note that, for $\abs{y}^\beta\geq \norm{u}_\infty$, we have 
$u(x_n+y)\leq \abs{y}^\beta$. Again, for $\abs{y}^\beta\leq \norm{u}_\infty$, we have
$$\dist(x_n+y, \Omega^c)\leq \dist(x_n, \Omega^c) +\abs{y}\leq r_n + |y|.$$
Applying \eqref{E1.9}, we have $u(x_n+y)\leq \kappa (r_n + |y|)^\beta\leq \kappa (r^\beta_n + |y|^\beta)$. Since $\beta<1<\gamma$,
\eqref{E1.10} follows.

On the other hand, since $2s\gamma=\beta(1+\delta)$,
\begin{align*}
-\cI_{\rm sym} w_n + r_n^{(2s-1)\gamma}H_s(\grad w_n)
&= r^{2s\gamma-\beta-\delta \beta}_n
\frac{k(x_n+r_n^\gamma y)}{(w_k(y))^\delta}
\\
&=
\left[\frac{r^{\beta\delta}_n}{(\dist(x_n+r_n^\gamma y)^{\beta\delta}}\right]
k(x_n+r_n^\gamma y)\left[\frac{(\dist(x_n+r_n^\gamma y)^{\beta}}{u(x_n + r^\gamma_n y)}\right]^\delta.
\end{align*}
Since from the first part $w_k$ is uniformly bounded from below on every compact set $K$ for large $n$, 
we have $\abs{k(x_n+r_n^\gamma y)/w^\delta_k}$ bounded. Also, 
$$\left|\frac{\dist(x_n+r_n^\gamma y)}{r_n}-1\right|=\left|\frac{\dist(x_n+r_n^\gamma y)-\dist(x_n,\Omega^c)}{r_n}\right| \leq r_n^{\gamma-1}|y|\to 0,$$
uniformly on compacts, as $n\to\infty$. Thus, using \eqref{E1.10} and standard regularity theory \cite{CS09,SS16}, we see that $\{w_n\}$ is locally
H\"{o}lder continuous, uniformly in $n$. Let $w_{n_k}\to w$, along some sub-sequence. It is evident that $w>0$
and $|w|\leq \kappa_1 (1+|y|^\beta)$. Furthermore, from the stability property of viscosity solution, we obtain
$$-\cI_{\rm sym} w = \frac{c}{w^\delta}\quad \text{in}\; \Rd,$$
where $c=\lim_{n_k\to\infty}k(x_{n_k})>0$, which we can assume to exists or chose a further subsequence. Now applying Theorem~\ref{T1.6} we see that such $w$ can not exist. Hence \eqref{E1.9} cannot hold. This completes the proof.
\end{proof}

\subsection*{Acknowledgement}
The authors thank the referee for his/her valuable comments.
Part of this project was done during a visit of A.B. and E.T. at Departamento de Matemáticas (DM) of Universidad T\'ecnica Federico Santa Mar\'{i}a at Valparaíso. The kind hospitality of the department is acknowledged. This research of A.B. is partly supported by a SwarnaJayanti fellowship SB/SJF/2020-21/03. A.Q. was partially supported by Fondecyt Grant 1231585. E.T. was supported by Fondecyt Grant 1201897, and CNPq Grants 408169 and 306022.


\begin{thebibliography}{77}
\bibitem{AGS18} Adimurthi, J. Giacomoni, and S. Santra. Positive solutions to a fractional equation with singular nonlinearity,
J. Differential Equations 265, no.4, 1191–1226, 2018

\bibitem{AGQ} S. Alarc\'on, J. Garc\'ia-Meli\'an, and A. Quaas. 
Nonexistence of positive supersolutions to some nonlinear elliptic problems,
 J. Math. Pures Appl. (9), 99(5):618–634, 2013

\bibitem{ATEJ} N. Alibaud, F. del Teso, J. Endal, and E. R. Jakobsen. The Liouville theorem and linear operators satisfying the maximum principle, J. Math. Pures Appl. (9) 142 , 229--242, 2020

\bibitem{AS}	N. Armstrong,  S.  Boyan,  Nonexistence of positive supersolutions of elliptic equations via the maximum principle,  Communications in Partial Differential Equations. 36 (2011), 2011-2047

\bibitem{Barles91}
G. Barles {\em A Weak Bernstein Method for Fully Nonlinear Elliptic Equations.} Diff. and Integral Equations, 4(2): 241-262, 1991.

\bibitem{Barles-12} G. Barles, E. Chasseigne, A. Ciomaga, and C. Imbert.
 Lipschitz regularity of solutions for mixed integro-differential equations,
  J. Differential Equations, 252(11):6012--6060, 2012


\bibitem{BCI08} G. Barles, E. Chasseigne, and C. Imbert. 
On the Dirichlet problem for second-order elliptic integro-differential equations. Indiana Univ. Math. J. 57(1), 213--246, 2008



\bibitem{BDGQ} B. Barrios, L. Del Pezzo, J. Garc\'ia-Meli\'an, and A. Quaas. 
A priori bounds and existence of solutions for some nonlocal elliptic problems,
 Rev. Mat. Iberoam. 34, 195--220, 2018

\bibitem{BD20} B. Barrios and L. M. Del Pezzo.
Study of the existence of supersolutions for nonlocal equations with gradient terms,
Milan J. Math. 88, no.2, 267--294, 2020

\bibitem{BBMP} B. Barrios, I. De Bonis, M. Medina, and I. Peral. 
Semilinear problems for the fractional laplacian with a singular nonlinearity,
 Open Math., 13, 390--407, 2015

\bibitem{Biswas} A. Biswas.
Liouville type results for systems of equations involving fractional Laplacian in exterior domains,
 Nonlinearity 32, 2246--2268, 2019

\bibitem{BT23} A. Biswas and E. Topp.
 Nonlocal ergodic control problem in $\mathbb{R}^d$,
Math. Annalen 390, 45--94, 2024

\bibitem{BK23} A. Biswas and S. Khan.
 Existence-Uniqueness for nonlinear integro-differential equations with drift in $\R^d$, 
SIAM J. Math. Anal. 55 (5), 4378--4409, 2023 


\bibitem{BS22} D. Berger and R. L. Schilling.
On the Liouville and strong Liouville properties for a class of non-local operators,
 Math. Scand. 128(2): 365--388, 2022

\bibitem{BGV} M-F. Bidaut-V\'eron, M. Garc\'ia-Huidobro, and L. V\'eron.
Estimates of solutions of elliptic equations with a source reaction term involving the product of the function and its gradient,
Duke Math. J. 168, no.8, 1487--1537, 2019

\bibitem{CDV} X. Cabr\'e, S. Dipierro, and E. Valdinoci.
 The Bernstein technique for integro-differential equations,
Arch. Ration. Mech. Anal., 243(3):1597--1652, 2022

\bibitem{CS09}
L. Caffarelli and L. Silvestre.
\newblock Regularity theory for fully nonlinear integro-differential equations,
Comm. Pure Appl. Math., 62(5):597--638, 2009

\bibitem{CGM24} F. Camilli, A. Goffi, and C. Mendico. Quantitative and qualitative properties for Hamilton-Jacobi PDEs via the nonlinear adjoint method. Ann. Sc. Norm. Super. Pisa Cl. Sci, 2024

\bibitem{CDLP10} I.Capuzzo Dolcetta, F.Leoni ,and A.Porretta. H\"{o}lder estimates for degenerate elliptic equations with coercive Hamiltonians. Trans. Am. Math. Soc., 362(9):4511--4536, 2010

\bibitem{CDM08} G. Caristi, L. D'Ambrosio and E. Mitidieri. Liouville theorems for some nonlinear inequalities.
Tr. Mat. Inst. Steklova 260, 2008

\bibitem{CM97} G. Caristi and E. Mitidieri. Nonexistence of positive solutions of quasilinear equations. Adv. Differential Equations, 2(3):319--359, 1997


\bibitem{CL09} W. Chen and C. Li.
 An integral system and the Lane-Emden conjecture, 
Discret. Contin. Dyn. Syst. 24, 1167–1184, 2009


\bibitem{CGT} A. Ciomaga, D. Ghilli, and E. Topp.
Periodic homogenization for weakly elliptic Hamilton-Jacobi- Bellman equations with critical fractional diffusion,
Comm. Partial Differential Equations, 47(1):1-38, 2022

\bibitem{CG21} M. Cirant and A. Goffi.
On the Liouville property for fully nonlinear equations with superlinear first-order terms,
"Proceedings of the Conference on Geometric and Functional Inequalities and Recent Topics in Nonlinear PDEs", Contemporary Mathematics, American Mathematical Society,  781 (2023)

\bibitem{CV12} P. Constantin and V. Vicol. Nonlinear maximum principles for dissipative linear nonlocal operators and applications,  Geometric and Functional Analysis
Volume 22, 1289--1321, 2012

\bibitem{CL00} A. Cutr\`i and F. Leoni.
On the Liouville property for fully nonlinear equations, 
Ann. Inst. H. Poincar\'e Anal. Non Lin\'eaire 17 219–245, 2000

\bibitem{DSV19} S. Dipierro, O. Savin, E. Valdinoci, Definition of fractional Laplacian for functions
with polynomial growth, Rev. Mat. Iberoam. 35, 1079-1122, 2019

\bibitem{Fall16} M. M. Fall.
Entire s-harmonic functions are affine,
Proc. Amer. Math. Soc.144, no.6, 2587--2592, 2016

\bibitem{FW16} M. M. Fall and T. Weth.
Liouville theorems for a general class of nonlocal operators,
Potential Anal. 45, no.1, 187--200, 2016

\bibitem{FS11} A. Farina and J. Serrin,
Entire solutions of completely coercive quasilinear elliptic equations II, 
J. Differential Equations 250, no. 12, 4409--4436, 2011

\bibitem{FQ}P. Felmer,  A. Quaas, Fundamental solutions and Liouville type theorems for nonlinear integral operators,  Adv.  Math. 226 (2011), 2712-2738.

\bibitem{FRRO} X. Fern\'{a}ndez-Real and X. Ros-Oton, 
Integro-differential elliptic equations, Progr. Math., 350, Birkhäuser/Springer, Cham, [2024],


\bibitem{FPS20} R. Filippucci, P. Pucci, and P. Souplet.
A Liouville-Type Theorem for an Elliptic Equation with Superquadratic Growth in the Gradient,
 Adv. Nonlinear Stud.  20(2): 245--251, 2020

\bibitem{FPR} R. Filippucci, P. Pucci, and M. Rigoli.
Nonlinear weighted p-Laplacian elliptic inequalities with gradient terms,
 Commun. Contemp. Math. 12 ,501--535, 2010

\bibitem{F09} R. Filippucci.
Nonexistence of positive weak solutions of elliptic inequalities,
Nonlinear Anal. ,2903--2916, 2009
 
 \bibitem{Goffi} A. Goffi.
 A priori Lipschitz estimates for nonlinear equations with mixed local and nonlocal diffusion via the adjoint-Bernstein method,
 Boll Unione Mat Ital, 2022. https://doi.org/10.1007/s40574-022-00340-w

\bibitem{Gui-Lin93} C. Gui and F. Lin.
Regularity of an elliptic problem with a singular nonlinearity,
Proc. Roy. Soc. Edinburgh, Sec. A, 123, 1021--1029, 1993


\bibitem{LM91} A. C. Lazer and P. J. McKenna.
On a singular nonlinear elliptic boundary-value problem,
 Proc. Amer. Math. Soc., 111, 721--730, 1991
 
\bibitem{IL90} H. Ishii and P. L. Lions.
Viscosity solutions of fully non-linear second-order elliptic partial differential equations, 
J. Differential Equations, 83, No.1, 26--78, 1990

\bibitem{Lions} P. L. Lions. 
Quelques remarques sur les probl\`{e}mes elliptiques quasilin\'{e}aires du second ordre,
 J. Anal. Math. 45, 234--254, 1985

\bibitem{MP04} E. Mitidieri and S. I. Pohozaev. Towards a unified approach to nonexistence of solutions for a
class of differential inequalities. Milan J. Math., 72, 129--162, 2004
 
 \bibitem{PS78} L. A. Peletier and J. Serrin.
  Gradient bounds and Liouville theorems for quasilinear elliptic equations,
   Ann. Scuola Norm. Sup. Pisa Cl. Sci. (4) 5, no. 1, 65--104, 1978

\bibitem{PP13} A. Porretta and E. Priola. Global Lipschitz regularizing effects for linear and nonlinear parabolic equations. J. Math. Pures Appl. (9), 100(5):633--686, 2013
 
 \bibitem{QX16} A. Quaas and A. Xia.
  A Liouville type theorem for Lane-Emden systems involving the fractional
Laplacian, Nonlinearity 29, 2279--2297, 2016

\bibitem{RS01} M. Rigoli and A. G. Setti. Liouville type theorems for $\phi$-subharmonic functions. Rev. Mat. Iberoamericana, 17(3):471--520, 2001

\bibitem{RS-Duke} X. Ros-Oton and J. Serra. 
Boundary regularity for fully nonlinear integro-differential equations,
Duke Math. J.165, no.11, 2079--2154, 2016

\bibitem{RS16} X. Ros-Oton and J. Serra.
Regularity theory for general stable operators,
J. Differential Equations 260, no.12, 8675--8715, 2016

\bibitem{SS16} R. W. Schwab and L. Silvestre.
 Regularity for parabolic integro-differential equations with very irregular kernels,
  Anal. PDE, 9(3):727--772, 2016

\bibitem{Ser15} J. Serra.
Regularity for fully nonlinear nonlocal parabolic equations with rough kernels,
Calc. Var. Partial Differential Equations, 54(1):615--629, 2015


\bibitem{Sil12a} Luis Silvestre.
On the differentiability of the solution to an equation with drift and fractional diffusion,
Indiana University Mathematics Journal
Vol. 61, No. 2, 557--584, 2012

\bibitem{Sil12b} L. Silvestre. 
H\"{o}lder estimates for advection fractional-diffusion equations, 
Annali della Scuola Normale Superiore di Pisa. Classe di Scienze, 2012

\bibitem{Yang} Y. Yang. 
Gradient Estimates for the Equation $\Delta u + c u^{-\alpha}=0$ on Riemannian Manifolds,
Acta Mathematica Sinica, Vol. 26,1177--1182,  2010

\bibitem{XY13} X. Yu. 
Liouville type theorems for integral equations and integral systems, 
Calc. Var. Partial Differential Equations 46, 75--95, 2013


\end{thebibliography}

\end{document}